\documentclass[final]{siamltex}

\usepackage{amssymb}
\usepackage{amsmath}
\usepackage{graphicx}
\usepackage{latexsym}
\usepackage{verbatim}
\usepackage{float}
\usepackage{graphics}
\usepackage{url}
\usepackage{subfigure}
\usepackage{multirow}
\usepackage{hyperref}
\usepackage{color}
\usepackage{pgfplots}
\usepgfplotslibrary{groupplots}

\usepackage[linesnumbered,ruled,vlined,nofillcomment]{algorithm2e}
\SetKwFor{Loop}{Loop}{}{EndLoop}
\SetKw{Break}{break}
\SetKwComment{tcblk}{--- }{ ---}
\SetCommentSty{textbf}

\newtheorem{Rem}{Remark}

\newcommand{\abs}[1]{\left\lvert #1 \right\rvert}
\newcommand{\norm}[1]{\left\lVert #1 \right\rVert}
\newcommand{\smashnorm}[1]{\left\lVert \smash[t]{#1} \right\rVert}
\newcommand{\inprod}[2]{\left( #1, #2 \right)}
\newcommand{\eye}{\mathrm{I}}
\newcommand{\ud}{\mathrm{d}}

\DeclareMathOperator*{\argmin}{arg\,min}

\newcommand{\Newp}{\alpha}
\newcommand{\Newq}{\beta}
\newcommand{\Newr}{\delta}
\newcommand{\News}{\zeta}
\newcommand{\Newt}{\eta}

\title{Double saddle-point preconditioning for Krylov methods in the inexact sequential homotopy method}

\author{John W. Pearson\thanks{School of Mathematics, The University of Edinburgh, James Clerk Maxwell Building, The King's Buildings, Peter Guthrie Tait Road, Edinburgh, EH9 3FD, United Kingdom ({\tt j.pearson@ed.ac.uk})}\and Andreas Potschka\thanks{Institute of Mathematics, Clausthal University of Technology, Erzstr.~1, 38678 Clausthal-Zeller\-feld, Germany ({\tt andreas.potschka@tu-clausthal.de})}}

\begin{document}
\maketitle

\begin{abstract}
We derive an extension of the sequential homotopy method that allows for the application of inexact solvers for the linear (double) saddle-point systems arising in the local semismooth Newton method for the homotopy subproblems. For the class of problems that exhibit (after suitable partitioning of the variables) a zero in the off-diagonal blocks of the Hessian of the Lagrangian, we propose and analyze an efficient, parallelizable, symmetric positive definite preconditioner based on a double Schur complement approach. For discretized optimal control problems with PDE constraints, this structure is often present with the canonical partitioning of the variables in states and controls. We conclude with numerical results for a badly conditioned and highly nonlinear benchmark optimization problem with elliptic partial differential equations and control bounds. The resulting method
  allows for the parallel solution of large 3D problems.
\end{abstract}



\pagestyle{myheadings}
\thispagestyle{plain}
\markboth{J. W. PEARSON AND A. POTSCHKA}{PRECONDITIONED INEXACT SEQUENTIAL HOMOTOPY METHOD}

\section{Introduction}\label{sec:Intro}

We are interested in approximately solving large-scale optimization problems of the form
\begin{equation}
  \label{eqn:OP}
  \min \phi({x}) \text{ over } {x} \in C \text{ subject to } c({x}) = 0,
\end{equation}
where the objective functional $\phi: X \to \mathbb{R}$ and the equality constraint $c: X \to Y$ are twice continuously differentiable and $X = \mathbb{R}^n$, $Y = \mathbb{R}^m$, equipped with inner products
\[
  \inprod{{u}}{{v}}_X = {u}^T M_X {v}, \quad
  \inprod{{u}}{{v}}_Y = {u}^T M_Y {v},
\]
for sparse symmetric positive definite matrices $M_X \in \mathbb{R}^{n \times n}$ and $M_Y \in \mathbb{R}^{m \times m}$.
We model inequality constraints by the nonempty closed convex set
\begin{equation} \label{eqn:upper_and_lower_bounds}
  C = \{ {x} \in \mathbb{R}^n \mid {x}^{\mathrm{l}}_i \le {x}_i \le {x}^{\mathrm{u}}_i \text{ for } i = 1, \dotsc, n \}
\end{equation}
with lower and upper variable bounds ${x}^{\mathrm{l}} < {x}^{\mathrm{u}} \in \overline{\mathbb{R}}^n$, whose entries may take on values of $\pm \infty$. Moreover, we assume that the constraint Jacobian $c'(x)$, the objective Hessian $\nabla^2 \phi(x)$, and all constraint Hessians $\nabla^2 c_i(x)$, $i=1, \dotsc, m$, are sparse matrices.

A sequential homotopy method has recently been proposed~\cite{PoBo21} for the approximate solution of a Hilbert space generalization of~\eqref{eqn:OP}, where the resulting linear saddle-point systems were numerically solved by direct methods based on sparse matrix decompositions. The aim of this article is to address the challenges that arise when the linear systems are solved only approximately by Krylov-subspace methods and to analyze and leverage novel preconditioners that exploit a multiple saddle-point structure, in particular double saddle-point form, which often arises in optimal control problems with partial differential equation (PDE) constraints \cite{MNN17,SoZu19,PePo21}. We restrict the presentation here to finite but high-dimensional $X$ and $Y$ to stay focused on the linear algebra issues.

The design principle of the sequential homotopy method is staying in a neighborhood of a suitably defined flow, which can be interpreted as the result of an idealized method with infinitesimal stepsize. Flows of this kind were first used by Davidenko~\cite{Davidenko1953} and later extended by various researchers as the basis for a plethora of globalization methods~\cite{Arrow1958,Deuflhard1974,Deuflhard1991,Hohmann1994,Deuflhard1998,Bock2000,Deuflhard2004,Potschka2016,Potschka2018,BGPSG20,PoBo21}, often with a focus on affine invariance principles and partly to infinite-dimensional problems. However, special care needs to be taken for most of these approaches when solving nonconvex optimization problems, as the Newton flow is attracted to saddle points and maxima. Here we extend the sequential homotopy method of~\cite{PoBo21}, which solves a sequence of projected backward Euler steps on a projected gradient/antigradient flow, to allow for inexact linear system solutions and hence the use of Krylov-subspace methods. Salient properties of the sequential homotopy method are that the difficulties of nonconvexity and constraint degeneracy are handled on the nonlinear level by an implicit regularization similar to regularized/stabilized Sequential Quadratic Programming~(see, e.g.,~\cite{Wr98,Ha99,GiRo13,GiKuRo17,GiKuRo17a}). If semismooth Newton methods (see, e.g.,~\cite{Mifflin1977,Qi1993,ChQiSu1998,Ito1990,Ul02,HiItKu2002,Hintermueller2004,Ito2004,Hintermueller2006,Hintermueller2010}) are applied, the resulting methods can handle the inequality constraints in an active-set fashion. Active-set methods are of high interest, especially when a sequence of related problems needs to be solved, because warm-starts that reuse the solution from a previously solved problem can be easily accomplished. For approaches based on interior-point methods (see, e.g.,~\cite{GoPoPe2022}), efficient warm-starting is still an unsolved issue.

We propose and analyze new preconditioners for the resulting linear system at each iteration of the sequential homotopy method. Preconditioning PDE-constrained optimization problems has been a subject of considerable interest of late (see, e.g.,~\cite{ScZu07,RDW10,Zu11,PW12,PSW12,PSW14,PST15,PG17}), however relatively few such methods have made use of the double saddle-point structure. We refer the reader to recent, related work \cite{MNN17,SoZu19,PePo21,BrGr21}. We provide theoretical results on preconditioners for double saddle-point systems and, having arranged the linear systems obtained from the active-set approach to this form, we present efficient and flexible approximations to be applied within a Krylov-subspace solver. Having explained each approximation step, we demonstrate the performance of a block-diagonal (and also a block-triangular) preconditioner on a highly nonlinear benchmark optimization problem, and present some results on potential parallelizability.

This paper is structured as follows. Section \ref{sec:SHM} states the sequential homotopy method and its Krylov-subspace active-set interpretation. In Section \ref{sec:InexactSemismooth}, we propose the use of an inexact semismooth Newton corrector in the sequential homotopy method and discuss the active-set-dependent linear systems which we need to solve. Section \ref{sec:Precon} gives results on preconditioners for double saddle-point matrices, which is the form of these systems, and Section \ref{sec:Implement} contains details of our implementation and approximations in our preconditioners. Section \ref{sec:Results} presents a range of numerical results.

\section{Sequential homotopy method}\label{sec:SHM}

We briefly recapitulate the sequential homotopy method proposed in~\cite{PoBo21}. Let
\[
  L^{\rho}(x, y) := \phi(x) + \frac{\rho}{2} c(x)^T M_Y^{-1} c(x) + y^T c(x)
\]
denote the augmented Lagrangian with some sufficiently large $\rho \ge 0$.\footnote{This form hinges on the interpretation that $c$ in fact maps to the dual space of $Y$: we have that $L^{\rho}(x, y) = \phi(x) + \frac{\rho}{2} [M_Y^{-1} c(x)]^T M_Y [M_Y^{-1} c(x)] + y^T M_Y [M_Y^{-1} c(x)] = \phi(x) + \frac{\rho}{2} \norm{\tilde{c}(x)}_Y + \inprod{y}{\tilde{c}(x)}_Y$ with the Riesz representation $\tilde{c}(x) = M_Y^{-1} c(x)$ defined by the variational equality $\inprod{y}{\tilde{c}(x)}_Y = y^T c(x) ~ \forall y \in Y$ (cf.~\cite{GuHeSa2014}). In the finite-dimensional space $Y = \mathbb{R}^m$, it might seem more elegant to simplify the arguments by identifying $Y$ with its dual and the duality pairing with an inner product. However, if~\eqref{eqn:OP} is a discretization of an infinite-dimensional problem, the distinction is necessary if we want to preserve our chances of ending up with a mesh-independent algorithm.}

The sequential homotopy method solves a sequence of subproblems that differ in $(\hat{x}, \hat{y}) \in C \times Y$ and a homotopy parameter $\Delta t > 0$:
\begin{equation}
  \label{eqn:hom}
  \begin{aligned}
    x - P_C\left( \hat{x} - \Delta t \nabla_x L^{\rho}(x, y) \right) &= 0,\\
    y - \phantom{P_C\,} \left( \hat{y} + \Delta t \nabla_y L^{\rho}(x, y) \right) &= 0,
  \end{aligned}
\end{equation}
for the unknowns $(x, y) \in X \times Y$, where $P_C$ is the (in general nonsmooth) projection onto $C$. By \cite[Thm.~4]{PoBo21}, we know that for sufficiently small $\Delta t > 0$, \eqref{eqn:hom} admits a unique solution provided that a Lipschitz condition on $\nabla L^{\rho}$ holds. For $\Delta t = 0$, the solution to~\eqref{eqn:hom} is $(x, y) = (\hat{x}, \hat{y})$. If a continuation of solutions $(x(\Delta t), y(\Delta t))$ to~\eqref{eqn:hom} exists for $\Delta t \to \infty$, then $x(\Delta t), y(\Delta t)$ tends to a critical point of the original problem~\eqref{eqn:OP}. If the solution $(x(\Delta t), y(\Delta t))$ cannot be continued (or if we decide not to follow it) further than some finite $\Delta t^\ast > 0$, we can update the reference point $(\hat{x}, \hat{y}) = (x(\Delta t^\ast), y(\Delta t^\ast))$ and recommence the next homotopy leg (hence the name sequential homotopy method).


Local semismooth Newton methods are ideal candidates to solve~\eqref{eqn:hom} (see~\cite{PoBo21}), because the solution is close to $(\hat{x}, \hat{y})$ if $\Delta t > 0$ is sufficiently small. 
From this vantage point, our investigation of \emph{inexact} semismooth Newton methods for~\eqref{eqn:hom} provides a class of active-set methods that are based on Krylov-subspace methods, which work particularly well in combination with suitable preconditioners such as those provided in Section~\ref{sec:Precon}.

\section{Homotopy method with inexact semismooth Newton corrector}\label{sec:InexactSemismooth}

Although a local semismooth Newton method could in principle be applied to~\eqref{eqn:hom} directly, we prefer to equivalently reformulate~\eqref{eqn:hom} to avoid numerical issues for large $\Delta t$ and to obtain symmetric linearized systems (up to nonsymmetric active set modifications). 
This approach is possible for general symmetric positive definite $M_X$, but we prefer to restrict the discussion to the important special case where $M_X$ is a diagonal matrix, which reduces the notational and computational burden significantly. We elaborate on this restriction in the context of finite element (FE) discretizations in Sec.~\ref{sec:Precon}.
In the case of diagonal $M_X$, the projection
$
  P_C(v) = \argmin_{x \in C} \norm{x - v}_X 
$
can be represented by entrywise clipping via
\begin{equation} \label{eqn:pointwise_projection}
  P_C(v)_i := \max \left(x^\mathrm{l}_i, \min \left(v_i, x^\mathrm{u}_i\right)\right) \quad \text{for } i = 1, \dotsc, n.
\end{equation}
To achieve a better scaling for large $\Delta t$, we premultiply the rows of equation system~\eqref{eqn:hom} by $\lambda M_X$ and $-\lambda M_Y$, where $\lambda = 1 / \Delta t$, which leads to the equivalent system
\begin{equation} \label{eqn:sym_hom}
  F(x, y; \hat{x}, \hat{y}, \lambda) =
  \begin{pmatrix}
    F_1(x, y; \hat{x}, \hat{y}, \lambda) \\
    F_2(x, y; \hat{x}, \hat{y}, \lambda)
  \end{pmatrix}
  :=
  \begin{pmatrix}
    \lambda M_X \left(x - P_C(\hat{x} - \Delta t \nabla_x L^\rho(x, y))\right) \\
    -\lambda M_Y (y - \hat{y}) + M_Y \nabla_y L^\rho(x, y)
  \end{pmatrix} = 0.
\end{equation}
It is apparent that $F(\cdot, \cdot \, ; \hat{x}, \hat{y}, \lambda)$ is Lipschitz continuous and piecewise continuously differentiable and, hence, semismooth~\cite{scholtes_2012_introduction,Ul02}.
In each iteration of a semismooth Newton method for solving~\eqref{eqn:sym_hom}, we first need to evaluate $F$ at the current iterate $(x,y)$. 
Denoting the diagonal entries of $M_X$ by $d_i > 0$, $i = 1, \dotsc, n$ and using that the gradient is the Riesz representation of the derivative, the argument $v$ of $P_C$ in~\eqref{eqn:sym_hom} consists of the entries
\[
  v_i = \hat{x}_i - \frac{\Delta t}{d_i} \frac{\ud}{\ud x_i} L^\rho(x, y).
\]
We then identify the lower and upper active sets by comparing $v$ with $x^\mathrm{l}$ and $x^\mathrm{u}$ as
\begin{equation}
  \begin{aligned}
    \mathbb{A}^\mathrm{l} &:= \left\{ i \mid x^\mathrm{l}_i \ge v_i \right\} = \left\{ i \mid \tfrac{\ud}{\ud x_i} L^\rho(x, y) \ge \lambda d_i (\hat{x}_i - x^\mathrm{l}_i) \right\}, \\
    \mathbb{A}^\mathrm{u} &:= \left\{ i \mid x^\mathrm{u}_i \le v_i \right\} = \left\{ i \mid \tfrac{\ud}{\ud x_i} L^\rho(x, y) \le \lambda d_i (\hat{x}_i - x^\mathrm{u}_i) \right\}.
  \end{aligned}
\end{equation}
This leads to the representations
\begin{align*}
  F_1(x, y; \hat{x}, \hat{y}, \lambda)_i &=
  \begin{cases}
    \lambda d_i (x_i - x^\mathrm{l}_i) & \text{for } i \in \mathbb{A}^\mathrm{l}, \\
    \lambda d_i (x_i - x^\mathrm{u}_i) & \text{for } i \in \mathbb{A}^\mathrm{u}, \\
    \lambda d_i (x_i - \hat{x}_i) + \tfrac{\ud}{\ud x_i} L^\rho(x, y) & \text{otherwise},
  \end{cases}\\
  F_2(x, y; \hat{x}, \hat{y}, \lambda) &= -\lambda M_Y (y - \hat{y}) + c(x).
\end{align*}
An element of the set-valued generalized derivative $\partial F(x, y; \hat{x}, \hat{y}, \lambda)$ can then conveniently be constructed by substitution of the active rows $i \in \mathbb{A} := \mathbb{A}^\mathrm{l} \cup \mathbb{A}^\mathrm{u}$ of the symmetric matrix
\[
  \lambda \diag(M_X, -M_Y) + \nabla^2 L^\rho(x, y)
  =
  \begin{pmatrix}
    \lambda M_X + H^\rho & G^T \\
    G & -\lambda M_Y
  \end{pmatrix}
\]
by rows consisting entirely of zeros except for diagonal entries of $\lambda d_i$, $i = 1, \dotsc, n$.
Here, we have used the abbreviations $G = c'(x)$, $H = \nabla^2 L^0(x, y + \rho M_Y^{-1} c(x))$ and write that $H^\rho = \nabla^2_{xx} L^\rho(x, y) = H + \rho G^T M_Y^{-1} G$.

We can then multiply all active rows $i \in \mathbb{A}$ and all active entries of the residual $-F(x, y; \hat{x}, \hat{y}, \lambda)$ by $\Delta t / d_i$, and permute all active rows and columns to the upper left of the linear system to arrive at
\begin{equation} \label{eqn:augm_lin_sys_with_act_set}
  \begin{pmatrix}
    \eye_{\abs{\mathbb{A}}} & 0 & 0 \\
    \eye_{\overline{\mathbb{A}}}^T \left[\lambda M_X + H^\rho \right] \eye_\mathbb{A} & \lambda \tilde{M}_X + \tilde{H}^\rho & \tilde{G}^T \\
    G \eye_\mathbb{A} & \tilde{G} & -\lambda M_Y
  \end{pmatrix}
  \begin{pmatrix}
    \Delta x_{\mathbb{A}} \\ \Delta x_{\overline{\mathbb{A}}} \\
    \Delta y
  \end{pmatrix}
  = -
  \begin{pmatrix}
    b_0 \\ b_1 \\ b_2
  \end{pmatrix},
\end{equation}
with $\overline{\mathbb{A}} = \{ 1, \dotsc, n\} \setminus \mathbb{A}$, $\eye_{\mathbb{A}}$ a rectangular matrix consisting of the columns $i \in \mathbb{A}$ of $\eye_n$ (likewise for $\eye_{\overline{\mathbb{A}}}$), $\tilde{M}_X = \eye_{\overline{\mathbb{A}}}^T M_X \eye_{\overline{\mathbb{A}}}$, $\tilde{H}^\rho = \eye_{\overline{\mathbb{A}}}^T H^\rho \eye_{\overline{\mathbb{A}}}$, and $\tilde{G} = G \eye_{\overline{\mathbb{A}}}$. 

The same nonlinear transformation as in~\cite[Sec.~5.1]{PoBo21} can be used to avoid forming the dense matrix $G^T M_Y^{-1} G$ in $H^\rho$ so that we only need to solve the sparse system
\begin{equation} \label{eqn:discr_system_with_active_set}
  \begin{pmatrix}
    \eye_{\abs{\mathbb{A}}} & 0 & 0 \\
    \eye_{\overline{\mathbb{A}}}^T \left[\lambda M_X + H \right] \eye_\mathbb{A} & \lambda \tilde{M}_X + \tilde{H} & \tilde{G}^T \\
    G \eye_\mathbb{A} & \tilde{G} & -\frac{\lambda}{1 + \rho \lambda} M_Y
  \end{pmatrix}
  \begin{pmatrix}
    \Delta x_{\mathbb{A}} \\ \Delta x_{\overline{\mathbb{A}}} \\ \Delta \tilde{y}
  \end{pmatrix}
  = -
  \begin{pmatrix}
    b_0 \\ b_1 \\ \tilde{b}_2
  \end{pmatrix},
\end{equation}
where $\tilde{b}_2 = \frac{1}{1 + \rho \lambda} b_2$ and $\tilde{H} = \eye_{\overline{\mathbb{A}}}^T H \eye_{\overline{\mathbb{A}}}$. The solution $\Delta y$ of~\eqref{eqn:augm_lin_sys_with_act_set} can then be recovered from $\Delta \tilde{y}$ via $\Delta y = \frac{1}{1 + \rho \lambda} (\Delta \tilde{y} + \rho M_Y^{-1} b_2 ).$


\begin{algorithm}[p]
  \caption{Inexact sequential homotopy method}
  \label{alg:iseqhom}
  \DontPrintSemicolon
  \KwData{$z_0 = (x_0, y_0) \in C \times Y$, $0 < \theta_\mathrm{ref} < \theta_\mathrm{max} < 1$, $0 < \lambda_\mathrm{min} \le \lambda_{\mathrm{term}}$, $0 < \lambda_{\mathrm{red}} < 1 < \lambda_{\mathrm{inc}}$, $K_P, K_I \ge 0$, $\mathrm{tol} > 0$}
  Initialize $x = x_0$, $y = y_0$, $\lambda = 1$, $I_{\mathrm{PI}} = 0$\;
  \For{$k = 0, 1, \dotsc$}{
    Update reference $\hat{x} = x$, $\hat{y} = y$, set accept $=$ false\;
    \tcblk{Preparations for semismooth Newton corrector steps}
    Compute $b_2 = c(x)$, Riesz rep.~$r = M_Y^{-1} b_2$, and $b_1 = \frac{\ud}{\ud x}L^0(x, y + \rho r)^T$\;
    Compute $\mathcal{H} = \nabla^2 L^0(x, y + \rho r)$\;
    \While{not \textnormal{accept}}{
      \tcblk{Semismooth Newton corrector step}
      Set $\mathbb{A}^\mathrm{l} = \{ i \mid (b_1)_i \ge \lambda d_i (\hat{x}_i - x^\mathrm{l}_i)\}$, $\mathbb{A}^\mathrm{u} = \{ i \mid (b_1)_i \le \lambda d_i (\hat{x}_i - x^\mathrm{u}_i)\}$\;
      Set $(\tilde{b}_1)_i = x_i - x^\mathrm{l/u}_i$ for $i \in \mathbb{A}$ and $(\tilde{b}_1)_i = (b_1)_i$ for $i \in \overline{\mathbb{A}}$\;
      Scale $\tilde{b}_2 = \frac{1}{1 + \rho \lambda} b_2$\;
      Compose $\hat{\mathcal{A}} = \diag \bigl( \lambda M_X, -\frac{\lambda}{1 + \rho \lambda} M_Y \bigr) + \mathcal{H}$\;
      Copy $\tilde{\mathcal{A}} = \hat{\mathcal{A}}$ and override rows $i \in \mathbb{A}$ of $\tilde{\mathcal{A}}$ with $e_i^T$\;
      Solve $\displaystyle \tilde{\mathcal{A}}
      \begin{pmatrix}
        \Delta x \\
        \Delta \tilde{y}  
      \end{pmatrix}
      = -
      \begin{pmatrix}
        \tilde{b}_1 \\
        \tilde{b}_2
      \end{pmatrix}$
      approximately\;
      Recover $\Delta y = \frac{1}{1 + \rho \lambda} (\Delta \tilde{y} + \rho r)$\;
      Project step $x^+ = P_C(x + \Delta x)$ and set $y^+ = y + \Delta y$\;
      \tcblk{Simplified semismooth Newton corrector step}
      Compute $b_2^+ = \lambda M_Y (\hat{y} - y^+) + c(x^+)$, $r^+ = M_Y^{-1} b_2^+$, $b_1^+ = \frac{\ud}{\ud x} L^0(x^+, y^+ + \rho r^+)^T$\;
      Set $\mathbb{A}^\mathrm{l} = \{ i \mid (b^+_1)_i \ge \lambda d_i (\hat{x}_i - x^\mathrm{l}_i)\}$, $\mathbb{A}^\mathrm{u} = \{ i \mid (b^+_1)_i \le \lambda d_i (\hat{x}_i - x^\mathrm{u}_i)\}$\;
      Set $(\tilde{b}^+_1)_i = x^+_i - x^\mathrm{l/u}_i$ for $i \in \mathbb{A}$, $(\tilde{b}^+_1)_i = \lambda d_i (x^+_i - \hat{x}_i) + (b^+_1)_i$ for $i \in \overline{\mathbb{A}}$\;
      Scale $\tilde{b}_2^+ = \frac{1}{1 + \rho \lambda} b_2^+$\;
      Copy $\tilde{\mathcal{A}} = \hat{\mathcal{A}}$ and override rows $i \in \mathbb{A}^+$ of $\tilde{\mathcal{A}}$ with $e_i^T$\;
      Solve $\displaystyle \tilde{\mathcal{A}}
      \begin{pmatrix}
        \Delta x^+ \\
        \Delta \tilde{y}^+  
      \end{pmatrix}
      = -
      \begin{pmatrix}
        \tilde{b}_1^+ \\
        \tilde{b}_2^+
      \end{pmatrix}$
      approximately\;
      Recover $\Delta y^+ = \frac{1}{1 + \rho \lambda} (\Delta \tilde{y}^+ + \rho r^+)$\;
      Project step $x^{++} = P_C(x^+ + \Delta x^+)$ and set $y^{++} = y^+ + \Delta y^+$\;
      \tcblk{Stepsize adjustment}
      Set $\theta = \left( \frac{\smashnorm{x^{++} - x^+}_X^2 + \smashnorm{y^{++} - y^+}_Y^2}{\smashnorm{x^+ - x}_Y^2 + \smashnorm{y^+ - y}_Y^2}\right)^{1/2}$\;
      \eIf(\textsl{(monotonicity test successful)}){$\theta \le \theta_\mathrm{max}$}{
        Set accept $=$ true, $x = x^{++}$, $y = y^{++}$\;
        \lIf{$\norm{x - \hat{x}}_X^2 + \norm{y - \hat{y}}_Y^2 \le \mathrm{tol}^2$ and $\lambda \le \lambda_\mathrm{term}$}{\Return $(x, y)$}
        \eIf(\textsl{(stepsize PI controller)}){$\theta > 0$}{
          Set $e = \log(\theta_\mathrm{ref}) - \log(\theta)$\;
          Set $\lambda = \max(\lambda\exp(-K_P e -K_I I_{\mathrm{PI}}), \lambda_\mathrm{min})$, update $I_{\mathrm{PI}} = I_{\mathrm{PI}} + e$\;
        }{
          Reduce $\lambda = \max(\lambda \cdot \lambda_\mathrm{red}, \lambda_\mathrm{min})$\;
        }
      }{
        Increase $\lambda = \lambda \cdot \lambda_\mathrm{inc}$\;
        \lIf{$I_{\mathrm{PI}} > 0$}{set $I_{\mathrm{PI}} = 0$}
      }
    }
  }
\end{algorithm}

When a symmetric Krylov-subspace method is started with an initial guess that solves the first block row of~\eqref{eqn:discr_system_with_active_set} exactly, then the generated Krylov subspace will only depend on the symmetric 2-by-2 block in the lower right, and symmetric Krylov-subspace methods can safely be used.
Hence, preconditioners are only required for the symmetric part of~\eqref{eqn:discr_system_with_active_set}.

We provide a pseudocode description of the inexact sequential homotopy method in Alg.~\ref{alg:iseqhom}. The algorithm employs only one inexact semismooth Newton step plus one simplified (i.e., without updating the system matrix except for active set changes) semismooth Newton step, and then simultaneously updates the reference point $(\hat{x}, \hat{y})$ and $\lambda$ in $F$.
Additionally, the permutation of the active rows and columns to the top-left block of the system is not performed explicitly.

As a practical implementation, we use \cite[Alg.~1]{PoBo21}, with the modification that the linear systems are solved approximately.
The choice of the relative termination tolerance 
$\kappa \ge 0$ for the preconditioned residual norm of
the Krylov-subspace method to approximately solve the linear systems involving $\tilde{\mathcal{A}}$ is a delicate issue here.
The choice should effect less accurate and thus less expensive linear system solves when far away from the solution. 
In contrast to the setting in, e.g., \cite{ChQiSu1998}, there are more involved restrictions on $\kappa$, because we are shooting at a moving target inside the homotopy method. Unfortunately, these restrictions are, at least so far, hard to quantify. Qualitatively, the size of the tolerance $\kappa$ must be balanced with the size of $\lambda$. Ultimately, we would like to have a small $\lambda$ to progress fast on the nonlinear level, but this would necessitate small $\kappa$ to stay within the (moving) region of local convergence. Hence, we propose the affine-linear-plus-saturation heuristic
\[
  \kappa = P_{[\kappa_{\mathrm{min}}, \kappa_{\mathrm{max}}]}\left(\kappa_{\mathrm{max}} + (\kappa_{\mathrm{min}} - \kappa_{\mathrm{max}})
  \tfrac{\lambda - \lambda^r_0}{\lambda^r_1 - \lambda^r_0}\right).
\]
We employ this heuristic in the numerical experiments with the values set to
  $\kappa_{\mathrm{min}} = 10^{-7},$
  $\kappa_{\mathrm{max}} = 10^{-3},$
  $\lambda^r_0 = 1,$
  $\lambda^r_1 = 10^{-7}.$
We further enforce $\lambda \ge \lambda_{\mathrm{min}} = 10^{-7}$ in the adaptive stepsize control.

The numerical criterion for evaluating whether we keep inside the region of local convergence is the same increment monotonicity test as described in~\cite[Alg.~1]{PoBo21}, which also serves as the basis for the adaptive stepsize control. 
The stepsize proportional-integral (PI) controller constants are the same as in \cite[Alg.~1]{PoBo21}.

\section{Preconditioned iterative methods for double saddle-point systems}\label{sec:Precon}

To derive efficient preconditioners, it is important to exploit additional problem structure. We take interest here in a special property of discretizations of a large class of PDE-constrained optimal control problems:
\begin{equation} \label{eqn:PDECO_class}
  \begin{alignedat}{3}
    \min~ & \frac{1}{2} \norm{u - u_{\mathrm{d}}}^2_V + \frac{\gamma}{2} \norm{q - q_\mathrm{d}}^2_Q\\
    \text{ over } & (q, u) \in C_Q \times U,\\
    \text{s.t.~} & c((u, q)) = c((u, 0)) + c_q((0, 0)) q = 0,
  \end{alignedat}
\end{equation}
for which the matrix $H$ is of block-diagonal form, because the control and state variables are separated in the objective function and the control enters the (possibly nonlinear) state equation only affine linearly. We assume that the discretized space $X = Q \times U$ comprises discretized controls $q \in C_Q \subset Q$ and (unconstrained) discretized states $u \in C_U = U$ with a tracking term, while $y \in Y$ are the adjoint variables. We can model control bounds in $C_Q$, which is of the form~\eqref{eqn:upper_and_lower_bounds}. 

The inner product matrix $M_X = \diag(M_Q, M_U)$ is a block-diagonal matrix, which typically consists of a mass matrix $M_Q$ for the controls and a stiffness matrix $M_U$ for the states. The inner product matrix $M_Y$ is typically also a stiffness matrix.
There are, however, good reasons to use only the diagonal of a mass matrix for $M_Q$ (mass lumping): spectral equivalence~\cite{Wathen87} leads to an equivalent norm with mesh-independent equivalence constants and, just like the pointwise formula for projections in $L^2$, the projector $P_C$ turns into an entrywise projection, which can be cheaply computed via~\eqref{eqn:pointwise_projection}.

Following the obvious extension of variable naming above, a symmetric permutation of~\eqref{eqn:discr_system_with_active_set} leads to
\begin{equation} \label{Ahk}
  \begin{pmatrix}
    \eye_{\abs{\mathbb{A}}} & 0 & 0 & 0 \\
    \ast & \lambda \tilde{M}_Q + \tilde{H}_Q & \tilde{G}_Q^T & 0 \\
    \ast & \tilde{G}_Q & -\frac{\lambda}{1 + \rho \lambda} M_Y & G_U \\
    \ast & 0 & G_U^T & \lambda M_U + H_U
  \end{pmatrix}
  \begin{pmatrix}
    \Delta {q}_{\mathbb{A}} \\
    \Delta {q}_{\overline{\mathbb{A}}} \\
    \Delta \tilde{{y}} \\
    \Delta {u}
  \end{pmatrix}
  = -
  \begin{pmatrix}
    {b}_0 \\
    {b}_{1,q} \\
    \frac{1}{1 + \rho \lambda} {b}_2 \\
    {b}_{1,u}
  \end{pmatrix}.
\end{equation}
We remark here that the regularization parameter $\gamma$ only enters in the $\tilde{H}_Q$ block.

In this section we provide some theory of preconditioning matrices of the form
\begin{equation}
\ \label{calA} \mathcal{A}=\left(\begin{array}{ccc}
A_1 & B_{1}^T & 0 \\ B_{1} & -A_2 & B_{2}^T \\ 0 & B_{2} & A_3 \\
\end{array}\right),
\end{equation}
which are frequently referred to as \emph{double saddle-point systems}. This is important, as the systems we are required to solve in this paper are of the form \eqref{calA}. Specifically, after (trivially) eliminating the first block-row in \eqref{Ahk}, we obtain a system of the form \eqref{calA} from the lower right 3-by-3 blocks of~\eqref{Ahk}. 
It is important that the proposed preconditioners are robust with respect to the augmented Lagrangian coefficient $\rho > 0$ and the proximity parameter $\lambda$, which may typically vary between $10^{-12}$ and $10^1$.

We highlight that there has been much previous work on preconditioners for double saddle-point systems, see for example \cite{GaHe00,MNN17,BeBe18,SoZu19}. In particular, \cite{SoZu19} provides a comprehensive description of eigenvalues of preconditioned double saddle-point systems on the continuous level (i.e., by the operators involved). What follows in this section is an analysis for double saddle-point systems, inspired by the logic of the proof given in \cite[Thm.~4]{MNN17}.

For our forthcoming analysis we will make use of the following result, which concerns eigenvalues for generalized saddle-point systems preconditioned by a block-diagonal matrix. The bounds described below can be found elsewhere in the literature: for instance, see \cite[Cor. 1]{AxNe06}, \cite[Lem. 2.2]{SiWa94}, \cite[Thm.~4]{Pe13}.
\begin{theorem}\label{2by2}
Let
\begin{equation*}
\ \mathcal{A}_2=\left(\begin{array}{cc}
A_1 & B^T \\ B & -A_2 \\
\end{array}\right),\quad\quad\mathcal{P}_2=\left(\begin{array}{cc}
A_1 & 0 \\ 0 & S_1 \\
\end{array}\right),
\end{equation*}
where $A_1$, $S_1=A_2+BA_1^{-1}B^T$ are assumed to be symmetric positive definite, and $A_2$ is assumed to be symmetric positive semi-definite. Then all eigenvalues of $\mathcal{P}_2^{-1}\mathcal{A}_2$ satisfy $\mu\left(\mathcal{P}_2^{-1}\mathcal{A}_2\right)\in\left[-1,\frac{1}{2}(1-\sqrt{5})\right]\cup\left[1,\frac{1}{2}(1+\sqrt{5})\right].$
\end{theorem}

We now analyze a block-diagonal preconditioner for matrix systems of the form \eqref{calA}, in the simplified setting that $A_{3}=0$, using the logic of \cite[Thm.~4]{MNN17}.
\begin{theorem}\label{3by3C=0}
Let
\begin{equation*}
\ \mathcal{A}_0=\left(\begin{array}{ccc}
A_1 & B_{1}^T & 0 \\ B_{1} & -A_2 & B_{2}^T \\ 0 & B_{2} & 0 \\
\end{array}\right),\quad\quad\mathcal{P}_D=\left(\begin{array}{ccc}
A_1 & 0 & 0 \\ 0 & S_1 & 0 \\ 0 & 0 & S_2 \\
\end{array}\right),
\end{equation*}
where $A_1$, $S_1=A_2+B_{1}A_1^{-1}B_{1}^T$, $S_2=B_{2}S_1^{-1}B_{2}^T$ are assumed to be symmetric positive definite, and $A_2$ is assumed to be symmetric positive semi-definite.\footnote{The matrix $B_{2}$ must therefore have at least as many columns as rows, and have full row rank.} Then all eigenvalues of $\mathcal{P}_D^{-1}\mathcal{A}_0$ satisfy:
\begin{align*}
\ \mu\left(\mathcal{P}_D^{-1}\mathcal{A}_0\right)\in{}&\left[-\frac{1}{2}(1+\sqrt{5}),2\hspace{0.15em}\emph{cos}\left(\frac{5\pi}{7}\right)\right]\cup\left[-1,\frac{1}{2}(1-\sqrt{5})\right] \\
\ &\quad\quad\cup\left[2\hspace{0.15em}\emph{cos}\left(\frac{3\pi}{7}\right),\frac{1}{2}(\sqrt{5}-1)\right]\cup\left[1,2\hspace{0.15em}\emph{cos}\left(\frac{\pi}{7}\right)\right],
\end{align*}
which to $3$ decimal places are $[-1.618,-1.247]\cup[-1,-0.618]\cup[0.445,0.618]\cup[1,1.802]$.
\end{theorem}
\begin{proof}
Examining the associated eigenproblem
\begin{equation*}
\ \left(\begin{array}{ccc}
A_1 & B_{1}^T & 0 \\ B_{1} & -A_2 & B_{2}^T \\ 0 & B_{2} & 0 \\
\end{array}\right)\left(\begin{array}{c}
x_1 \\ x_2 \\ x_3 \\
\end{array}\right)=\mu\left(\begin{array}{ccc}
A_1 & 0 & 0 \\ 0 & S_1 & 0 \\ 0 & 0 & S_2 \\
\end{array}\right)\left(\begin{array}{c}
x_1 \\ x_2 \\ x_3 \\
\end{array}\right),
\end{equation*}
we obtain that
\begin{subequations}
\begin{align}
\ \label{A3eig1} A_1x_1+B_{1}^Tx_2={}&\mu{}A_1x_1, \\
\ \label{A3eig2} B_{1}x_1-A_2x_2+B_{2}^Tx_3={}&\mu{}S_1x_2, \\
\ \label{A3eig3} B_{2}x_2={}&\mu{}S_2x_3.
\end{align}
\end{subequations}
We now prove the result by contradiction, that is we assume there exists an eigenvalue outside the stated intervals. From \eqref{A3eig1}, we see that
\begin{equation}
\ \label{pA13} \Newp A_1x_1+B_{1}^Tx_2=0\quad\Rightarrow\quad\Newp x_1=-A_1^{-1}B_{1}^Tx_2,
\end{equation}
where $\Newp=1-\mu\neq0$ by assumption. Substituting \eqref{pA13} into \eqref{A3eig2} tells us that
\begin{align}
\ \nonumber 0={}&\Newp \left[B_{1}x_1-(1+\mu)A_2x_2+B_{2}^Tx_3-\mu{}B_{1}A_1^{-1}B_{1}^Tx_2\right] \\
\ \nonumber ={}&-B_{1}A_1^{-1}B_{1}^Tx_2-\Newp (1+\mu)A_2x_2+\Newp B_{2}^Tx_3-\mu\Newp B_{1}A_1^{-1}B_{1}^Tx_2 \\
\ \label{sum2terms3} ={}&-[\Newq B_{1}A_1^{-1}B_{1}^T+\Newr A_2]x_2+\Newp B_{2}^Tx_3,
\end{align}
where $\Newq=1+\mu\Newp =1+\mu-\mu^2\neq0$, $\Newr=\Newp(1+\mu)=1-\mu^2\neq0$ by assumption. Now, for the cases $\mu\in(-\infty,-1)$, $\mu\in(\frac{1}{2}(1-\sqrt{5}),1)$, and $\mu\in(\frac{1}{2}(1+\sqrt{5}),+\infty)$, $\Newq$ and $\Newr$ have the same signs (negative, positive, and negative, respectively), so $\Newq B_{1}A_1^{-1}B_{1}^T+\Newr A_2$ is a definite (and hence invertible) matrix. Thus, in these cases, \eqref{sum2terms3} tells us that
\begin{equation*}
\ x_2=\Newp[\Newq B_{1}A_1^{-1}B_{1}^T+\Newr A_2]^{-1}B_{2}^Tx_3,
\end{equation*}
which we may then substitute into \eqref{A3eig3} to yield that
\begin{equation}
\ \label{pCCT} 0=\Newp B_{2}[\Newq B_{1}A_1^{-1}B_{1}^T+\Newr A_2]^{-1}B_{2}^Tx_3-\mu{}B_{2}[B_{1}A_1^{-1}B_{1}^T+A_2]^{-1}B_{2}^Tx_3.
\end{equation}

We now highlight that $\mu\neq0$; otherwise \eqref{pCCT} would read that
\begin{equation*}
\ 0=B_{2}[B_{1}A_1^{-1}B_{1}^T+A_2]^{-1}B_{2}^Tx_3=S_2x_3\quad\Rightarrow\quad x_3=0,
\end{equation*}
which means that
\begin{equation*}
\ \left(\begin{array}{cc}
A_1 & B_{1}^T \\ B_{1} & -A_2 \\
\end{array}\right)\left(\begin{array}{c}
x_1 \\ x_2 \\
\end{array}\right)=\mu\left(\begin{array}{cc}
A_1 & 0 \\ 0 & S_1 \\
\end{array}\right)\left(\begin{array}{c}
x_1 \\ x_2 \\
\end{array}\right),
\end{equation*}
holds (along with $B_{2}x_2=0$). Applying Thm.~\ref{2by2} then gives that $\mu\in[-1,\frac{1}{2}(1-\sqrt{5})]\cup[1,\frac{1}{2}(1+\sqrt{5})]$, yielding a contradiction. Using that $\mu\neq0$, as well as that $\Newp\neq0$ by assumption, we may divide \eqref{pCCT} by $\Newp\mu$ to obtain that
\begin{align*}
\ 0={}&B_{2}[\mu\Newq B_{1}A_1^{-1}B_{1}^T+\mu\Newr A_2]^{-1}B_{2}^Tx_3-B_{2}[\Newp B_{1}A_1^{-1}B_{1}^T+\Newp A_2]^{-1}B_{2}^Tx_3 \\
\ ={}&B_{2}[\News B_{1}A_1^{-1}B_{1}^T+\Newt A_2]^{-1}B_{2}^Tx_3-B_{2}[\Newp B_{1}A_1^{-1}B_{1}^T+\Newp A_2]^{-1}B_{2}^Tx_3,
\end{align*}
where $\News=\mu\Newq =\mu+\mu^2-\mu^3$ and $\Newt=\mu\Newr=\mu-\mu^3$. As the situation $x_3=0$ reduces the problem to that of Thm.~\ref{2by2}, where $\mu$ is contained within a subset of the intervals claimed here, we may reduce the analysis to the case $x_3\neq0$ and write
\begin{equation}
\ \label{x3Phix3} 0=x_3^{T}B_{2}\left([\News B_{1}A_1^{-1}B_{1}^T+\Newt A_2]^{-1}-[\Newp B_{1}A_1^{-1}B_{1}^T+\Newp A_2]^{-1}\right)B_{2}^Tx_3=:x_3^{T}\Psi x_3,
\end{equation}
with both $B_{1}A_1^{-1}B_{1}^T$ and $A_2$ symmetric positive semi-definite. We now consider different cases (we have already excluded the possibility that $\mu=0$):
\begin{itemize}
\item \underline{$\mu\in(-\infty,-\frac{1}{2}(1+\sqrt{5})):$} Here $\Newp,\News,\Newt>0$, such that $\News>\Newp$, $\Newt>\Newp$. Hence $\News B_{1}A_1^{-1}B_{1}^T+\Newt A_2\succ\Newp B_{1}A_1^{-1}B_{1}^T+\Newp A_2$ (meaning $[\News B_{1}A_1^{-1}B_{1}^T+\Newt A_2]-[\Newp B_{1}A_1^{-1}B_{1}^T+\Newp A_2]$ is positive definite), and $\Psi$ is negative definite. This yields a contradiction with \eqref{x3Phix3}, so there is no $\mu\in(-\infty,-\frac{1}{2}(1+\sqrt{5}))$.
\item \underline{$\mu\in(2\hspace{0.15em}\text{cos}(\frac{5\pi}{7}),-1):$} Here $\Newp,\News,\Newt>0$, such that $\News<\Newp$, $\Newt<\Newp$. Hence $\News B_{1}A_1^{-1}B_{1}^T+\Newt A_2\prec\Newp B_{1}A_1^{-1}B_{1}^T+\Newp A_2$ (meaning $[\News B_{1}A_1^{-1}B_{1}^T+\Newt A_2]-[\Newp B_{1}A_1^{-1}B_{1}^T+\Newp A_2]$ is negative definite), and $\Psi$ is positive definite. Similarly to above, this yields a contradiction.
\item \underline{$\mu\in(\frac{1}{2}(1-\sqrt{5}),0):$} Here, $\Newp>0$, $\News,\Newt<0$, so $[\News B_{1}A_1^{-1}B_{1}^T+\Newt A_2]^{-1}$ is negative definite and $[\Newp B_{1}A_1^{-1}B_{1}^T+\Newp A_2]^{-1}$ is positive definite. Therefore $\Psi$ is negative definite, yielding a contradiction.
\item \underline{$\mu\in(0,2\hspace{0.15em}\text{cos}(\frac{3\pi}{7})):$} Here $\Newp,\News,\Newt>0$, such that $\News<\Newp$, $\Newt<\Newp$. Hence $\News B_{1}A_1^{-1}B_{1}^T+\Newt A_2\prec\Newp B_{1}A_1^{-1}B_{1}^T+\Newp A_2$, and $\Psi$ is positive definite, yielding a contradiction.
\item \underline{$\mu\in(\frac{1}{2}(\sqrt{5}-1),1):$} Here $\Newp,\News,\Newt>0$, such that $\News>\Newp$, $\Newt>\Newp$. Hence $\News B_{1}A_1^{-1}B_{1}^T+\Newt A_2\succ\Newp B_{1}A_1^{-1}B_{1}^T+\Newp A_2$, and $\Psi$ is negative definite, yielding a contradiction.
\item \underline{$\mu\in(2\hspace{0.15em}\text{cos}(\frac{\pi}{7}),+\infty):$} Here $\Newp,\News,\Newt<0$, such that $\News<\Newp$, $\Newt<\Newp$. Hence $\News B_{1}A_1^{-1}B_{1}^T+\Newt A_2$ and $\Newp B_{1}A_1^{-1}B_{1}^T+\Newp A_2$ are both negative definite, such that $\News B_{1}A_1^{-1}B_{1}^T+\Newt A_2\prec\Newp B_{1}A_1^{-1}B_{1}^T+\Newp A_2$. Therefore $[\News B_{1}A_1^{-1}B_{1}^T+\Newt A_2]^{-1}\succ[\Newp B_{1}A_1^{-1}B_{1}^T+\Newp A_2]^{-1}$, and $\Psi$ is positive definite, yielding a contradiction.
\end{itemize}
The result is thus proved by contradiction.
\end{proof}

By applying the structure of the proof of Thm.~\ref{3by3C=0}, we may analyze the analogous block-diagonal preconditioner for \eqref{calA}, for the more general case that $A_{3}$ is symmetric positive semi-definite. The following retrieves the result shown in Thm.~3.3 of the recent paper \cite{BrGr21}, using a different structure of proof:
\begin{theorem}\label{3by3}
Let
\begin{equation*}
\ \mathcal{P}_D=\left(\begin{array}{ccc}
A_1 & 0 & 0 \\ 0 & S_1 & 0 \\ 0 & 0 & S_2 \\
\end{array}\right),
\end{equation*}
with $\mathcal{A}$ as defined in \eqref{calA}, where $A_1$, $S_1=A_2+B_{1}A_1^{-1}B_{1}^T$, $S_2=A_3+B_{2}S_1^{-1}B_{2}^T$ are assumed to be symmetric positive definite, and $A_2$, $A_3$ are assumed to be symmetric positive semi-definite. Then all eigenvalues of $\mathcal{P}_D^{-1}\mathcal{A}$ satisfy:
\begin{equation*}
\ \mu\left(\mathcal{P}_D^{-1}\mathcal{A}\right)\in\left[-\frac{1}{2}(1+\sqrt{5}),\frac{1}{2}(1-\sqrt{5})\right]\cup\left[2\hspace{0.15em}\emph{cos}\left(\frac{3\pi}{7}\right),2\hspace{0.15em}\emph{cos}\left(\frac{\pi}{7}\right)\right],
\end{equation*}
which to $3$ decimal places are $[-1.618,-0.618]\cup[0.445,1.802]$.
\end{theorem}

\begin{Rem}
The above result guarantees a fixed rate of convergence for preconditioned {\scshape minres} applied to systems of the form \eqref{calA} with the properties stated. In \cite{BrGr21}, the authors also demonstrate the effect of approximating the blocks $A_1$, $S_1$, $S_2$ within the block-diagonal preconditioner, further highlighting the effectiveness of this strategy. We also highlight \cite[Thm.~3.2]{BrGr21}, which analyzes the case $A_{2}=0$.
\end{Rem}

Along with block-diagonal preconditioners for \eqref{calA}, we may construct block-triangular preconditioners, stated and analyzed in the simple result below.
\begin{theorem}\label{BTPrec}
Let
\begin{equation*}
\ \mathcal{P}_L=\left(\begin{array}{ccc}
A_1 & 0 & 0 \\ B_{1} & -S_1 & 0 \\ 0 & B_{2} & S_2 \\
\end{array}\right),\quad\quad\mathcal{P}_U=\left(\begin{array}{ccc}
A_1 & B_{1}^T & 0 \\ 0 & -S_1 & B_{2}^T \\ 0 & 0 & S_2 \\
\end{array}\right),
\end{equation*}
with $\mathcal{A}$ as defined in \eqref{calA}, where $A_1$, $S_1=A_2+B_{1}A_1^{-1}B_{1}^T$, $S_2=A_3+B_{2}S_1^{-1}B_{2}^T$ are assumed to be invertible. Then all eigenvalues of $\mathcal{P}_L^{-1}\mathcal{A}$ and $\mathcal{P}_U^{-1}\mathcal{A}$ are equal to $1$.
\end{theorem}
\begin{proof}
It may easily be verified that
\begin{equation*}
\ \mathcal{P}_L^{-1}\mathcal{A}=\left(\begin{array}{ccc}
\eye & A_1^{-1}B_{1}^T & 0 \\ 0 & \eye & -S_1^{-1}B_{2}^T \\ 0 & 0 & \eye \\
\end{array}\right),\quad\quad\mathcal{A}\mathcal{P}_U^{-1}=\left(\begin{array}{ccc}
\eye & 0 & 0 \\ B_{1}A_1^{-1} & \eye & 0 \\ 0 & -B_{2}S_1^{-1} & \eye \\
\end{array}\right),
\end{equation*}
which gives the result (with the result for $\mathcal{P}_U^{-1}\mathcal{A}$ inferred using similarity of this matrix and $\mathcal{A}\mathcal{P}_U^{-1}$).
\end{proof}

We note Thm.~\ref{BTPrec} makes no assumptions on the properties of $B_{2}$, or the positive definiteness of $A_1$, $S_1$, $S_2$, unlike Thms. \ref{3by3C=0} and \ref{3by3}. However, we note the limitations of Thm.~\ref{BTPrec}: with a block-triangular preconditioner, diagonalizability of the preconditioned system is not given, and implementing this preconditioner requires a nonsymmetric iterative solver for which the eigenvalue distribution does not guarantee a fixed convergence rate.

It would also be possible to combine $P_D$, $P_L$, and $P_U$ into one symmetric positive definite preconditioner $P_{LDU} = P_L P_D^{-1} P_U$, an approach which we investigate separately in~\cite{PePo21} for the more general case of multiple saddle-point systems of block-tridiagonal form. The parallel implementation of the preconditioner $P_{LDU}$ would also be possible with the same libraries that we use in Section~\ref{sec:Implement}, but at the expense of a more involved code. Because we focus on how to use indirect linear algebra methods inside the sequential homotopy method here, we restrict the numerical results to the use of the simpler preconditioners $P_D$ and $P_L$.

\section{Implementation details}\label{sec:Implement}

Among all problems of the form~\eqref{eqn:PDECO_class}, we focus our attention in the remainder of this work to a family of nonlinear and possibly badly conditioned benchmark problems of the specific form (cf.~\cite{Lu17,PoBo21}):
\begin{equation}
  \label{eqn:model_problem}
  \begin{alignedat}{3}
    \min~&\frac{1}{2} \int_{\Omega} \abs{u - u_{\mathrm{d}}}^2 + \frac{\gamma}{2} \int_{\Omega} \abs{q}^2 \quad
    && \text{over } u \in H^1_0(\Omega), q \in L^2(\Omega)\\
    \text{s.t.} ~& \nabla \cdot \left(\left[a + b \abs{u}^2\right] \nabla
    u\right) = q, \quad
    q_{\mathrm{l}} \le q \le q_{\mathrm{u}},
  \end{alignedat}
\end{equation}
where $\Omega \subset \mathbb{R}^d$, $d=2, 3$, is a bounded domain with
Lipschitz boundary and $a, b, \gamma > 0$ with control bounds
$q_{\mathrm{l}}, q_{\mathrm{u}} \in L^s(\Omega), s \in (2, \infty]$, and a tracking target function $u_{\mathrm{d}} \in L^2(\Omega)$. The difficulty of problem~\eqref{eqn:model_problem} can be tuned by the scalar parameters: smaller $\gamma$ and $a$ result in worse conditioning of the problem, while larger $b$ increases the nonlinearity. We caution that the interplay with the objective function is nontrivial: small values of $a$ in combination with large values of $b$ might lead to an optimal solution with small $|u|$ over $\Omega$, eventually resulting in less nonlinearity in the neighborhood of the optimal solution. We focus on $\gamma = 10^{-6}$, $a = 10^{-2}$, $b = 10^{2}$, which was experienced as the most difficult instance in the numerical results in~\cite{PoBo21}. All remaining problem data was set to the values in~\cite{PoBo21}.

After discretization of problem~\eqref{eqn:model_problem} with P1 finite elements, the application of Alg.~\ref{alg:iseqhom} leads to a sequence of linear systems of the form~\eqref{Ahk}. Eliminating the first block row, we arrive at a double saddle-point system of the form~\eqref{calA}, with
\begin{align*}
\ &A_1=\lambda \tilde{M}_Q + \tilde{H}_Q, ~~ A_2=\frac{\lambda}{1 + \rho \lambda} M_Y, ~~ A_3=\lambda M_U + H_U, ~~ B_1=\tilde{G}_Q, ~~ B_2=G_U^T,\\
\ &S_1=\frac{\lambda}{1 + \rho \lambda} M_Y+B_1 (\lambda \tilde{M}_Q + \tilde{H}_Q)^{-1}B_1^T, \quad S_2=\lambda M_U + H_U+B_2 S_1^{-1} B_2^T.
\end{align*}
In order to efficiently apply the block preconditioners derived above, we need to approximately apply the inverse operations of $A_1$, $S_1$, $S_2$ to generic vectors, in a computationally efficient way. We note that an alternative strategy would be to apply preconditioners to a re-ordered system of (classical) saddle-point form (see \cite{BGL05,MGW00}), however for problems with the structures considered, the leading block of the matrix is likely to possess a more problematic structure, with implications on its preconditioned approximation as well as that of the resulting Schur complement.

Numerically, we compare a number of increasingly efficient preconditioner choices, which we elaborate on in more detail below:
\begin{remunerate}
  \item \emph{Direct} sparse decomposition of $\mathcal{A}$ without the use of Krylov-subspace solver, which can be considered as a preconditioner that leads to convergence in one step.
  \item \emph{Basic Schur complement} block-diagonal preconditioner with sparse decomposition of $A_1$, $\hat{S}_1$, $\hat{S}_2$, where $\hat{S}_i$ approximates $S_i$. For empty active sets, the approximations are exact.
  \item \emph{Matching Schur complement} block-diagonal preconditioner with sparse decomposition of $A_1$ and $\hat{S}_1$ in combination with a matching approach for $\hat{S}_2$ using sparse decompositions for the factors of the multiplicative approximation of $\hat{S}_2$.
  \item \emph{AMG[$\hat{S}_1$] matching Schur complement} block-diagonal preconditioner, which coincides with the matching Schur complement preconditioner, except using an Algebraic Multigrid (AMG) approximation for $\hat{S}_1^{-1}$ instead of a direct decomposition.
  \item \emph{Decomposition-free Schur complement} block-diagonal preconditioner with completely decom\-position-free approximations of all Schur complements $A_1$, $S_1$, $S_2$ based on matching for $\hat{S}_2$. The preconditioner can be implemented using publicly-available, parallelizable preconditioners.
  \item \emph{Block-triangular} decomposition-free Schur complement preconditioner. In contrast to the previous preconditioners, this preconditioner must be used with a nonsymmetric solver such as {\scshape gmres} instead of {\scshape minres} because it is a block lower-triangular and not a symmetric preconditioner.
\end{remunerate}

Each preconditioner in the list above serves as the baseline for the next in terms of number of nonlinear iterations, which are expected to grow from top to bottom, and the resulting computation time. 
We use FEniCS~\cite{AlnaesBlechta2015a,LoggWellsEtAl2012a} to discretize problem~\eqref{eqn:model_problem} with P1 finite elements and implemented all preconditioners in PETSc~\cite{petsc-user-ref,petsc-web-page,petsc-efficient}.

We do not provide further details for preconditioner choices 1 and 6 because the application of a direct sparse decomposition is straightforward and choice 6 is just the block-triangular version of the block-diagonal preconditioner of choice 5. In the remainder of this section, we elaborate on the choices 2--5.

\subsection{Basic Schur complement preconditioner}

We first recall the inner product matrices: the matrix $M_Q$ is the diagonal part of a FE mass matrix $M_{\mathrm{FE}}$.
The matrices $M_U = M_Y$ are FE stiffness matrices. All of these can be easily assembled.
Thus, $\tilde{M}_Q = \eye_{\overline{\mathbb{A}}}^T M_Q \eye_{\overline{\mathbb{A}}}$, 
$\tilde{H}_Q = \gamma \eye_{\overline{\mathbb{A}}}^T M_{\mathrm{FE}} \eye_{\overline{\mathbb{A}}}$, 
and $\tilde{G}_Q = M_{\mathrm{FE}} \eye_{\overline{\mathbb{A}}}$. For the first Schur complement we obtain the explicit expressions
\begin{align*}
  A_1 &= \lambda \tilde{M}_Q + \tilde{H}_Q, \quad B_1 = \tilde{G}_Q, \quad A_2 = \frac{\lambda}{1 + \rho \lambda} M_Y, \\ 
  S_1 &= \frac{\lambda}{1 + \rho \lambda} M_Y + M_{\mathrm{FE}} \eye_{\overline{\mathbb{A}}} \left( \eye_{\overline{\mathbb{A}}}^T (\lambda M_Q + \gamma M_{\mathrm{FE}}) \eye_{\overline{\mathbb{A}}} \right)^{-1} \eye_{\overline{\mathbb{A}}}^T M_{\mathrm{FE}}.
\end{align*}
We observe that for $\lambda = 0$, the Schur complement $S_1$ has a nontrivial kernel spanned by $M_{\mathrm{FE}}^{-1} \eye_{\mathbb{A}}$, while it is equal to $(1/\gamma) M_{\mathrm{FE}} \eye_{\overline{\mathbb{A}}}$ on the orthogonal complement, which is spanned by $\eye_{\overline{\mathbb{A}}}$. Based on these comments, we propose approximating $M_{\mathrm{FE}}$ on each occasion in $S_1$ by its diagonal $M_Q$ (independently of the active set $\mathbb{A}$) to derive a basic Schur complement preconditioner using the approximate Schur complement
\[
  \hat{S}_1 = \frac{\lambda}{1 + \rho \lambda} M_Y + \frac{1}{\lambda + \gamma} M_Q.
\]
We remark that the stiffness term $M_Y$ dominates $\hat{S}_1$ for large $\lambda$ and for finer meshes with fixed (but possibly small) $\lambda$.
We can use direct sparse decompositions of $A_1$ and $\hat{S}_1$ to compute the action of their inverses to high accuracy.

For the basic Schur complement preconditioner, we approximate $S_2$ using the approximate first Schur complement
\[
  \hat{S}_2 = A_2 + B_2 \hat{S}_1^{-1} B_2^T.
\]
Regarding the explicit structure of $\hat{S}_2$, we first note that $B_2$ is a nonsymmetric FE matrix for the linearized elliptic PDE operator with respect to the state.
The direct assembly of $\hat{S}_2$ is prohibitive because $\hat{S}_2$ is a dense matrix. However, the action of $\hat{S}_2^{-1}$ can be computed by unrolling the Schur complement through a direct sparse decomposition of the sparse block matrix
\[
  \begin{pmatrix}
    \lambda M_U + H_U & B_2 \\
    B_2^T & -\hat{S}_1
  \end{pmatrix}.
\]
We investigate this preconditioner purely to assess the quality of the Schur complement approach in general and as a baseline for the following Schur complement approximations. We discourage the use of the basic Schur complement preconditioner for the solution of~\eqref{eqn:model_problem}.

\subsection{Matching Schur complement preconditioner}
We now address the approximation of the most delicate Schur complement $\hat{S}_2$ by a matching strategy, which has found considerable utility for PDE-constrained optimization problems \cite{PSW12,PW12}, including those with additional bound constraints \cite{PG17,PSW14}. A key challenge is that inverses of matrix sums are much harder to apply than inverses of matrix products.

For the problem at hand, the second Schur complement has the explicit form
\begin{align*}
  \hat{S}_2
  &= \lambda M_U + H_U + B_2 \hat{S}_1^{-1} B_2^T\\
  &= \lambda M_U + H_U + B_2 \left[ (1 + \rho \lambda)^{-1} \lambda M_Y + (\lambda + \gamma)^{-1} M_Q \right]^{-1} B_2^T.
\end{align*}
We now approximate $\hat{S}_2$ by a product $D \hat{S}_1^{-1} D^T$, where
$D = \alpha_1 M_U + \alpha_2 M_{\mathrm{FE}} + B_2$ such that
\[
\left( \alpha_1 M_U + \alpha_2 M_{\mathrm{FE}} \right) \hat{S}_1^{-1} \left( \alpha_1 M_U + \alpha_2 M_{\mathrm{FE}} \right) \approx \lambda {M}_U + {H}_U,
\]
and we wish to \emph{match} the factors $\alpha_1, \alpha_2 > 0$. As $H_U=M_{\mathrm{FE}}+N_U$, where $N_U$ contains nonlinear terms arising from the PDE operators, a computationally efficient choice is to neglect the effect of $N_U$, and match the terms in $M_Y$, $M_{\mathrm{FE}}$ separately. This heuristic strategy means we wish to select
\begin{equation*}
\begin{array}{rcl}
\left( \alpha_1 M_U \right) \left[ (1 + \rho \lambda)^{-1} \lambda M_Y \right]^{-1} \left( \alpha_1 M_U \right) = \lambda M_U & \Rightarrow & \alpha_1 = \frac{\lambda}{\sqrt{1 + \rho \lambda}}, \\
\left( \alpha_2 M_{\mathrm{FE}} \right) \left[ (\lambda + \gamma)^{-1} M_Q \right]^{-1} \left( \alpha_2 M_{\mathrm{FE}} \right) \approx M_{\mathrm{FE}} & \Rightarrow & \alpha_2 = \frac{1}{\sqrt{\lambda + \gamma}},
\end{array}
\end{equation*}
exploiting that $M_U=M_Y$ and $M_Q$ is the diagonal part of $M_\mathrm{FE}$.
Note that we can readily assemble $D = \lambda (1 + \rho \lambda)^{-1/2} M_U + (\lambda + \gamma)^{-1/2} M_{\mathrm{FE}} + B_2$ and use a direct sparse decomposition of $D$ to approximate $\hat{S}_2^{-1}$ with $\smash[t]{\hat{S}}_2^{-1} = D^{-T} \hat{S}_1 D^{-1}$.

We highlight that the structure of certain terms within $\hat{S}_2$, and hence the effectiveness of the resulting approximations, are dependent on the problem structure considered. For instance, the matrix $B_2^T$ is dependent on the form of the (linearized) differential operator, and the regularization terms within the objective function will lead to rank-deficient matrices when (for instance) boundary control problems are considered; these structures will then arise within the preconditioner. We believe the complicated, highly nonlinear benchmark problem considered provides a challenging and realistic test for our proposed preconditioning strategy.

\subsection{AMG[$\hat{S}_1$] matching Schur complement preconditioner}
This approach extends the matching Schur complement preconditioner by using an AMG preconditioner to approximate $\hat{S}_1$. We use \texttt{hypre/Boo\-merAMG}~\cite{FaYa02,HeYa02} with two sweeps of a V-cycle and one Jacobi iteration for each pre- and post-smoothing step.

\subsection{Decomposition-free Schur complement preconditioner}

This preconditioner further extends the AMG[$\hat{S}_1$] matching Schur complement preconditioner.
The matrix $A_1$ is a FE mass matrix, so may be well approximated by its diagonal \cite{Wathen87}. The use of a nested Conjugate Gradient ({\scshape cg}) solver for $A_1$ would require the use of a flexible outer Krylov-subspace solver such as flexible {\scshape gmres}~\cite{Sa93} due to the nonlinearity of {\scshape cg}. Instead, we apply a fixed number of 15 linear Chebyshev semi-iterations (see \cite{GoVa61a,GoVa61b,WR09}) on the diagonally-scaled $A_1$ with the optimal spectral bounds $[\frac{1}{2}, 2]$ (2D) and $[\frac{1}{2}, \frac{5}{2}$] (3D) from~\cite{Wathen87}, which are independent of the value of $\lambda$. 
Finally we also employ an AMG preconditioner for approximating the inverses of $D$ and $D^T$. We use \texttt{hypre/BoomerAMG} again with 4 V-cycle sweeps with two Jacobi sweeps for pre- and post-smoothing and a fixed relaxation weight of $0.7$. It is crucial to disable CF-relaxation for \texttt{BoomerAMG} to work correctly for application of the transposed preconditioner (see~\cite[p.~133]{hypre_documentation}).

\section{Numerical results}\label{sec:Results}

We solve problem~\eqref{eqn:model_problem} using the sequential homotopy method with the algorithmic parameters from~\cite{PoBo21}, with the changes described above and the Schur complement approximations of Section~\ref{sec:Implement}, applying block-diagonally preconditioned {\scshape minres} and block-lower-triangularly preconditioned {\scshape gmres}. A maximum number of 200 Krylov method iterations is used. When there is no convergence within this number of iterations, it is an indication that the current matrix $\hat{S}_2$ might be indefinite. In this case, the nonlinear step is flagged as failed and the sequential homotopy method increases $\lambda$, which eventually renders $\hat{S}_2$ positive definite. For the simplified semismooth Newton step, we prescribe the same number of Krylov method iterations that were adaptively determined by the standard termination criterion in the preceeding semismooth Newton step. When a sequential homotopy iteration fails due to excess of Krylov method iterations or due to violation of the monotonicity test, we mark the iteration with a black cross in the following figures.\footnote{All results were computed on a 2$\times$32-core AMD EPYC 7452 workstation with 256 GB RAM running GNU/Linux.}

We perform numerical experiments for P1 FE discretizations on triangular meshes of the unit square/cube with a varying number $N$ of elements per side.

\begin{figure}[bp]
  \centering
\begin{tikzpicture}

\definecolor{crimson2143940}{RGB}{214,39,40}
\definecolor{darkgray176}{RGB}{176,176,176}
\definecolor{darkorange25512714}{RGB}{255,127,14}
\definecolor{forestgreen4416044}{RGB}{44,160,44}
\definecolor{lightgray204}{RGB}{204,204,204}
\definecolor{steelblue31119180}{RGB}{31,119,180}

\begin{axis}[
height=0.3\textheight,
legend cell align={left},
legend style={
  fill opacity=0.8,
  draw opacity=1,
  text opacity=1,
  at={(0.97,0.03)},
  anchor=south east,
  draw=lightgray204
},
tick align=outside,
tick pos=left,
width=\textwidth,
x grid style={darkgray176},
xlabel={Outer nonlinear iteration \(\displaystyle k\)},
xmin=-5.2, xmax=109.2,
xtick style={color=black},
y grid style={darkgray176},
ylabel={Inner Krylov method iterations},
ymin=-4.05, ymax=85.05,
ytick style={color=black}
]
\addplot [semithick, steelblue31119180, mark=*, mark size=2, mark options={solid}]
table {%
0 3
1 5
2 9
3 10
4 14
5 16
6 18
7 21
8 23
9 17
10 29
11 29
12 31
13 22
14 36
15 46
16 34
17 31
18 32
19 61
20 63
21 57
22 55
23 50
24 58
25 58
26 36
27 65
28 36
29 66
30 36
31 36
32 36
33 34
34 34
35 34
36 34
37 36
38 36
39 34
40 41
41 63
42 34
43 37
44 37
45 36
46 36
47 34
48 38
49 68
50 42
51 35
52 35
53 36
54 31
55 33
56 33
57 28
58 39
59 70
60 65
61 65
62 63
};
\addlegendentry{$N = 40$}
\addplot [semithick, black, mark=x, mark size=2, mark options={solid}, only marks, forget plot]
table {%
21 57
36 34
39 34
};
\addplot [semithick, darkorange25512714, mark=*, mark size=2, mark options={solid}]
table {%
0 3
1 5
2 10
3 10
4 12
5 17
6 14
7 21
8 25
9 18
10 29
11 33
12 26
13 22
14 29
15 27
16 29
17 30
18 29
19 31
20 63
21 63
22 65
23 65
24 66
25 67
26 68
27 33
28 33
29 33
30 33
31 33
32 34
33 38
34 38
35 38
36 40
37 38
38 43
39 36
40 41
41 38
42 36
43 36
44 36
45 35
46 35
47 33
48 32
49 34
50 34
51 36
52 31
53 79
54 24
55 25
56 27
57 30
58 30
59 33
60 73
61 78
62 71
63 70
64 69
65 69
};
\addlegendentry{$N = 80$}
\addplot [semithick, black, mark=x, mark size=2, mark options={solid}, only marks, forget plot]
table {%
32 34
35 38
39 36
55 25
56 27
57 30
58 30
59 33
};
\addplot [semithick, forestgreen4416044, mark=*, mark size=2, mark options={solid}]
table {%
0 3
1 5
2 9
3 12
4 12
5 14
6 14
7 23
8 19
9 19
10 18
11 20
12 28
13 26
14 21
15 24
16 24
17 26
18 26
19 26
20 29
21 63
22 65
23 65
24 66
25 66
26 67
27 31
28 30
29 31
30 31
31 30
32 31
33 31
34 32
35 36
36 32
37 30
38 30
39 31
40 32
41 32
42 32
43 30
44 24
45 30
46 31
47 31
48 32
49 34
50 81
51 70
52 67
53 65
54 65
55 65
};
\addlegendentry{$N = 160$}
\addplot [semithick, black, mark=x, mark size=2, mark options={solid}, only marks, forget plot]
table {%
32 31
};
\addplot [semithick, crimson2143940, mark=*, mark size=2, mark options={solid}]
table {%
0 3
1 6
2 7
3 7
4 9
5 10
6 12
7 14
8 17
9 17
10 18
11 21
12 20
13 21
14 19
15 21
16 21
17 22
18 23
19 23
20 24
21 27
22 65
23 65
24 65
25 67
26 67
27 67
28 26
29 26
30 26
31 26
32 27
33 44
34 53
35 45
36 41
37 46
38 47
39 49
40 47
41 43
42 40
43 37
44 32
45 24
46 31
47 37
48 27
49 0
50 0
51 0
52 10
53 12
54 7
55 9
56 10
57 10
58 12
59 14
60 17
61 20
62 23
63 31
64 37
65 44
66 49
67 34
68 32
69 32
70 34
71 31
72 31
73 33
74 44
75 40
76 55
77 39
78 41
79 40
80 32
81 33
82 33
83 33
84 31
85 35
86 30
87 28
88 27
89 27
90 22
91 30
92 28
93 33
94 26
95 26
96 29
97 32
98 39
99 35
100 74
101 69
102 69
103 68
104 68
};
\addlegendentry{$N = 320$}
\addplot [semithick, black, mark=x, mark size=2, mark options={solid}, only marks, forget plot]
table {%
33 44
35 45
36 41
37 46
38 47
39 49
40 47
41 43
42 40
43 37
44 32
47 37
49 0
50 0
51 0
53 12
74 44
76 55
77 39
96 29
};
\end{axis}

\end{tikzpicture}
  \vspace{-1.5em}
  \caption{Krylov solver iterations: {\scshape minres} with the basic Schur complement preconditioner. The case $N=640$ is omitted due to excessive runtime.}
  \label{fig:ksp_iters_minres_exact_S1_S2}
\end{figure}
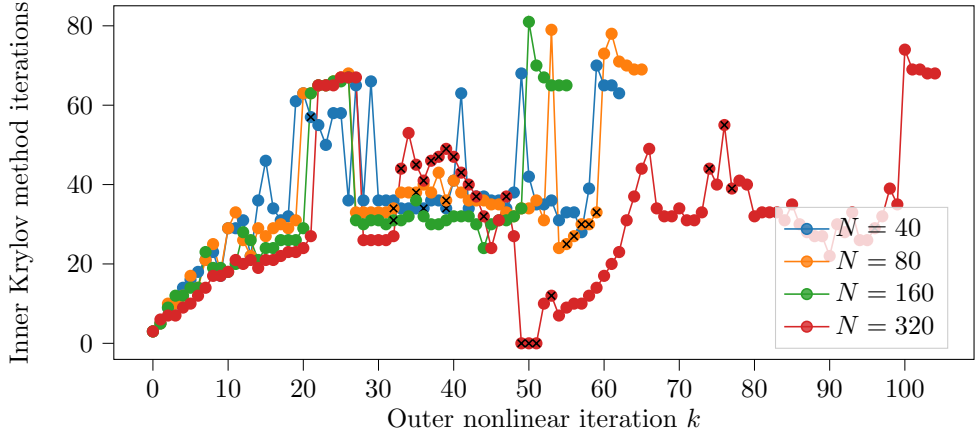

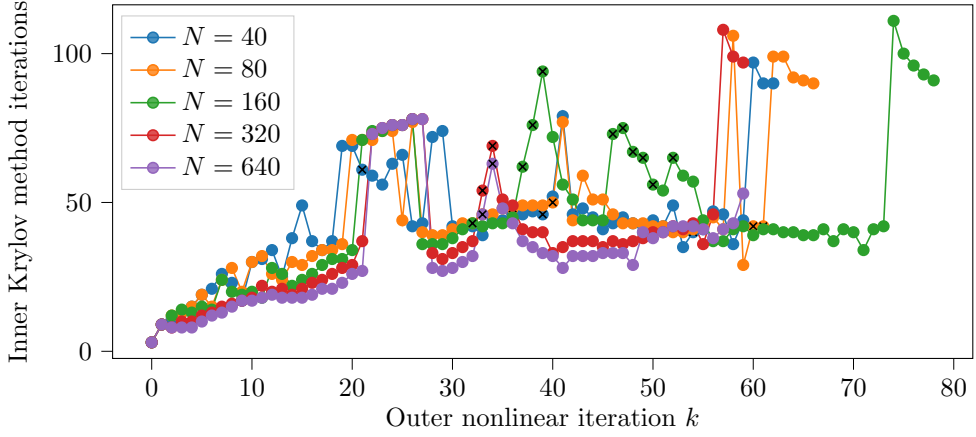
\begin{figure}[bp]
  \centering
\begin{tikzpicture}

\definecolor{crimson2143940}{RGB}{214,39,40}
\definecolor{darkgray176}{RGB}{176,176,176}
\definecolor{darkorange25512714}{RGB}{255,127,14}
\definecolor{forestgreen4416044}{RGB}{44,160,44}
\definecolor{lightgray204}{RGB}{204,204,204}
\definecolor{mediumpurple148103189}{RGB}{148,103,189}
\definecolor{steelblue31119180}{RGB}{31,119,180}

\begin{axis}[
height=0.3\textheight,
legend cell align={left},
legend style={
  fill opacity=0.8,
  draw opacity=1,
  text opacity=1,
  at={(0.01,0.97)},
  anchor=north west,
  draw=lightgray204
},
tick align=outside,
tick pos=left,
width=\textwidth,
x grid style={darkgray176},
xlabel={Outer nonlinear iteration \(\displaystyle k\)},
xmin=-3.9, xmax=81.9,
xtick style={color=black},
y grid style={darkgray176},
ylabel={Inner Krylov method iterations},
ymin=-2.4, ymax=116.4,
ytick style={color=black}
]
\addplot [semithick, steelblue31119180, mark=*, mark size=2, mark options={solid}]
table {%
0 3
1 9
2 11
3 13
4 15
5 19
6 21
7 26
8 23
9 17
10 30
11 31
12 34
13 25
14 38
15 49
16 37
17 34
18 37
19 69
20 69
21 61
22 59
23 56
24 63
25 66
26 42
27 43
28 72
29 74
30 42
31 42
32 42
33 39
34 44
35 44
36 46
37 46
38 47
39 46
40 52
41 79
42 46
43 48
44 45
45 41
46 43
47 45
48 43
49 40
50 44
51 42
52 49
53 35
54 40
55 41
56 47
57 46
58 36
59 44
60 97
61 90
62 90
};
\addlegendentry{$N = 40$}
\addplot [semithick, black, mark=x, mark size=2, mark options={solid}, only marks, forget plot]
table {%
21 61
36 46
39 46
54 40
};
\addplot [semithick, darkorange25512714, mark=*, mark size=2, mark options={solid}]
table {%
0 3
1 9
2 12
3 13
4 15
5 19
6 15
7 24
8 28
9 20
10 30
11 32
12 26
13 24
14 30
15 29
16 32
17 34
18 34
19 36
20 71
21 71
22 71
23 74
24 74
25 44
26 77
27 40
28 39
29 39
30 41
31 43
32 43
33 46
34 46
35 48
36 48
37 49
38 49
39 49
40 50
41 77
42 44
43 59
44 51
45 51
46 46
47 43
48 43
49 43
50 42
51 42
52 40
53 40
54 41
55 43
56 45
57 41
58 106
59 29
60 42
61 42
62 99
63 99
64 92
65 91
66 90
};
\addlegendentry{$N = 80$}
\addplot [semithick, black, mark=x, mark size=2, mark options={solid}, only marks, forget plot]
table {%
32 43
35 48
36 48
40 50
60 42
61 42
};
\addplot [semithick, forestgreen4416044, mark=*, mark size=2, mark options={solid}]
table {%
0 3
1 9
2 12
3 14
4 13
5 15
6 14
7 24
8 20
9 19
10 20
11 18
12 28
13 26
14 22
15 24
16 26
17 29
18 31
19 31
20 34
21 71
22 74
23 74
24 76
25 76
26 78
27 36
28 36
29 36
30 38
31 41
32 43
33 42
34 43
35 43
36 45
37 62
38 76
39 94
40 72
41 56
42 51
43 44
44 44
45 44
46 73
47 75
48 67
49 65
50 56
51 54
52 65
53 59
54 57
55 44
56 37
57 37
58 41
59 42
60 39
61 41
62 41
63 40
64 40
65 39
66 39
67 41
68 37
69 41
70 40
71 34
72 41
73 42
74 111
75 100
76 96
77 93
78 91
};
\addlegendentry{$N = 160$}
\addplot [semithick, black, mark=x, mark size=2, mark options={solid}, only marks, forget plot]
table {%
32 43
37 62
38 76
39 94
46 73
47 75
48 67
49 65
50 56
52 65
};
\addplot [semithick, crimson2143940, mark=*, mark size=2, mark options={solid}]
table {%
0 3
1 9
2 8
3 10
4 10
5 12
6 13
7 15
8 16
9 17
10 18
11 22
12 20
13 21
14 19
15 21
16 23
17 24
18 26
19 28
20 29
21 37
22 73
23 75
24 76
25 76
26 78
27 78
28 33
29 31
30 33
31 35
32 37
33 54
34 69
35 51
36 49
37 41
38 40
39 40
40 33
41 35
42 37
43 37
44 37
45 35
46 37
47 36
48 37
49 38
50 40
51 40
52 42
53 41
54 43
55 36
56 46
57 108
58 99
59 97
};
\addlegendentry{$N = 320$}
\addplot [semithick, black, mark=x, mark size=2, mark options={solid}, only marks, forget plot]
table {%
33 54
34 69
};
\addplot [semithick, mediumpurple148103189, mark=*, mark size=2, mark options={solid}]
table {%
0 3
1 9
2 8
3 8
4 8
5 10
6 12
7 13
8 15
9 17
10 17
11 18
12 19
13 18
14 18
15 18
16 19
17 21
18 21
19 23
20 26
21 27
22 73
23 75
24 76
25 76
26 78
27 78
28 28
29 27
30 28
31 30
32 32
33 46
34 63
35 48
36 43
37 37
38 35
39 33
40 32
41 28
42 32
43 32
44 32
45 33
46 33
47 33
48 29
49 40
50 38
51 40
52 42
53 42
54 42
55 41
56 38
57 41
58 43
59 53
};
\addlegendentry{$N = 640$}
\addplot [semithick, black, mark=x, mark size=2, mark options={solid}, only marks, forget plot]
table {%
33 46
34 63
};
\end{axis}

\end{tikzpicture}
  \vspace{-1.5em}
  \caption{Krylov solver iterations: {\scshape minres} with the matching Schur complement preconditioner.}
  \label{fig:ksp_iters_minres_lu_S1_match_S2}
\end{figure}    

\begin{figure}[tbp]
  \centering
\begin{tikzpicture}

\definecolor{crimson2143940}{RGB}{214,39,40}
\definecolor{darkgray176}{RGB}{176,176,176}
\definecolor{darkorange25512714}{RGB}{255,127,14}
\definecolor{forestgreen4416044}{RGB}{44,160,44}
\definecolor{lightgray204}{RGB}{204,204,204}
\definecolor{mediumpurple148103189}{RGB}{148,103,189}
\definecolor{steelblue31119180}{RGB}{31,119,180}

\begin{axis}[
height=0.3\textheight,
legend cell align={left},
legend style={
  fill opacity=0.8,
  draw opacity=1,
  text opacity=1,
  at={(0.01,0.97)},
  anchor=north west,
  draw=lightgray204
},
tick align=outside,
tick pos=left,
width=\textwidth,
x grid style={darkgray176},
xlabel={Outer nonlinear iteration \(\displaystyle k\)},
xmin=-4.7, xmax=98.7,
xtick style={color=black},
y grid style={darkgray176},
ylabel={Inner Krylov method iterations},
ymin=-3.3, ymax=135.3,
ytick style={color=black}
]
\addplot [semithick, steelblue31119180, mark=*, mark size=2, mark options={solid}]
table {%
0 3
1 9
2 13
3 13
4 15
5 19
6 22
7 27
8 24
9 17
10 30
11 34
12 35
13 25
14 38
15 51
16 37
17 34
18 37
19 69
20 71
21 61
22 59
23 56
24 64
25 66
26 42
27 43
28 74
29 76
30 42
31 42
32 42
33 39
34 44
35 44
36 44
37 46
38 49
39 48
40 55
41 85
42 46
43 42
44 43
45 43
46 43
47 41
48 44
49 40
50 42
51 39
52 43
53 39
54 41
55 39
56 99
57 86
58 83
59 82
60 81
61 81
};
\addlegendentry{$N = 40$}
\addplot [semithick, black, mark=x, mark size=2, mark options={solid}, only marks, forget plot]
table {%
21 61
36 44
39 48
};
\addplot [semithick, darkorange25512714, mark=*, mark size=2, mark options={solid}]
table {%
0 3
1 9
2 13
3 15
4 15
5 19
6 15
7 25
8 30
9 20
10 31
11 33
12 27
13 24
14 30
15 29
16 32
17 34
18 34
19 36
20 71
21 73
22 73
23 76
24 76
25 44
26 78
27 39
28 39
29 39
30 41
31 44
32 48
33 45
34 44
35 46
36 49
37 58
38 56
39 60
40 57
41 62
42 73
43 79
44 51
45 43
46 38
47 89
48 61
49 45
50 24
51 22
52 19
53 39
54 46
55 51
56 46
57 44
58 67
59 44
60 49
61 44
62 44
63 45
64 44
65 44
66 41
67 43
68 49
69 50
70 51
71 48
72 56
73 103
74 102
75 29
76 40
77 37
78 37
79 40
80 43
81 43
82 44
83 88
84 83
85 79
86 75
87 78
88 88
89 103
90 104
91 105
92 103
93 99
94 94
};
\addlegendentry{$N = 80$}
\addplot [semithick, black, mark=x, mark size=2, mark options={solid}, only marks, forget plot]
table {%
32 48
37 58
39 60
40 57
41 62
42 73
43 79
44 51
45 43
47 89
48 61
49 45
69 50
76 40
77 37
78 37
79 40
80 43
81 43
82 44
83 88
84 83
85 79
};
\addplot [semithick, forestgreen4416044, mark=*, mark size=2, mark options={solid}]
table {%
0 3
1 9
2 13
3 14
4 15
5 15
6 14
7 26
8 19
9 20
10 20
11 20
12 28
13 26
14 21
15 24
16 26
17 29
18 31
19 31
20 34
21 71
22 74
23 76
24 76
25 78
26 79
27 37
28 36
29 36
30 38
31 41
32 43
33 42
34 41
35 47
36 48
37 48
38 41
39 40
40 40
41 40
42 39
43 39
44 39
45 41
46 40
47 40
48 42
49 42
50 112
51 21
52 23
53 99
54 96
55 93
56 91
57 89
58 89
};
\addlegendentry{$N = 160$}
\addplot [semithick, black, mark=x, mark size=2, mark options={solid}, only marks, forget plot]
table {%
32 43
36 48
51 21
52 23
};
\addplot [semithick, crimson2143940, mark=*, mark size=2, mark options={solid}]
table {%
0 3
1 9
2 9
3 10
4 10
5 12
6 13
7 15
8 16
9 18
10 18
11 22
12 21
13 21
14 19
15 21
16 22
17 24
18 26
19 28
20 29
21 37
22 73
23 74
24 76
25 77
26 78
27 80
28 33
29 31
30 33
31 36
32 38
33 67
34 77
35 119
36 78
37 57
38 49
39 43
40 48
41 92
42 90
43 129
44 108
45 101
46 93
47 88
48 71
49 60
50 47
51 39
52 41
53 39
54 39
55 38
56 40
57 34
58 85
59 116
60 88
61 96
62 91
63 81
64 93
65 78
66 56
67 47
68 38
69 37
70 38
71 38
72 36
73 40
74 37
75 37
76 36
77 36
78 36
79 38
80 38
81 40
82 40
83 42
84 42
85 41
86 43
87 111
88 102
89 100
90 98
};
\addlegendentry{$N = 320$}
\addplot [semithick, black, mark=x, mark size=2, mark options={solid}, only marks, forget plot]
table {%
33 67
35 119
36 78
41 92
42 90
43 129
44 108
46 93
47 88
58 85
59 116
60 88
61 96
62 91
64 93
};
\addplot [semithick, mediumpurple148103189, mark=*, mark size=2, mark options={solid}]
table {%
0 3
1 9
2 8
3 8
4 8
5 10
6 12
7 13
8 15
9 17
10 18
11 18
12 19
13 18
14 18
15 18
16 19
17 21
18 21
19 23
20 26
21 27
22 74
23 76
24 76
25 76
26 78
27 78
28 28
29 27
30 28
31 30
32 33
33 56
34 77
35 64
36 43
37 37
38 37
39 35
40 32
41 32
42 32
43 32
44 27
45 34
46 33
47 35
48 35
49 37
50 32
51 39
52 42
53 34
54 28
55 21
56 32
57 39
58 39
59 41
60 75
61 31
62 31
63 37
64 41
65 56
66 103
67 97
68 95
69 94
};
\addlegendentry{$N = 640$}
\addplot [semithick, black, mark=x, mark size=2, mark options={solid}, only marks, forget plot]
table {%
33 56
34 77
55 21
56 32
57 39
62 31
63 37
};
\end{axis}

\end{tikzpicture}
  \vspace{-1.5em}
  \caption{Krylov solver iterations: {\scshape minres} with the decomposition-free Schur complement preconditioner.}
  \label{fig:ksp_iters_minres_lu_free}
\end{figure}
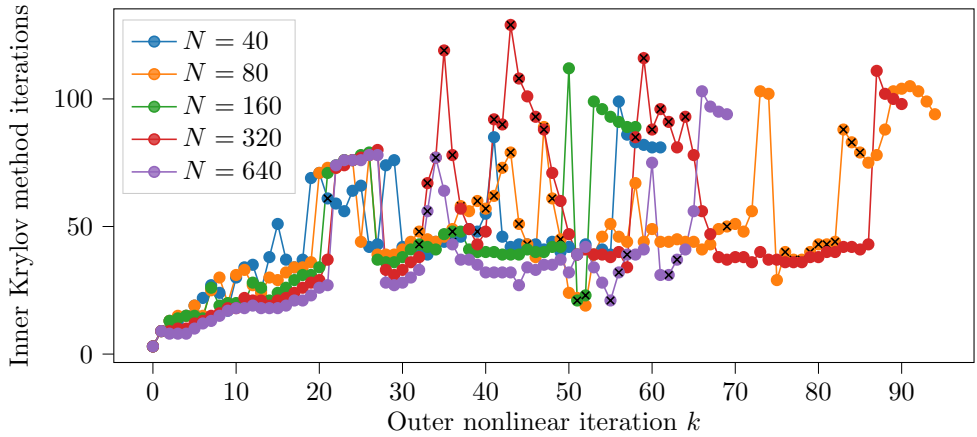

\begin{figure}[tbp]
  \centering
\begin{tikzpicture}

\definecolor{crimson2143940}{RGB}{214,39,40}
\definecolor{darkgray176}{RGB}{176,176,176}
\definecolor{darkorange25512714}{RGB}{255,127,14}
\definecolor{forestgreen4416044}{RGB}{44,160,44}
\definecolor{lightgray204}{RGB}{204,204,204}
\definecolor{mediumpurple148103189}{RGB}{148,103,189}
\definecolor{steelblue31119180}{RGB}{31,119,180}

\begin{axis}[
height=0.3\textheight,
legend cell align={left},
legend style={fill opacity=0.8, draw opacity=1, text opacity=1, at={(0.01,0.97)}, anchor=north west, draw=lightgray204},
tick align=outside,
tick pos=left,
width=\textwidth,
x grid style={darkgray176},
xlabel={Outer nonlinear iteration \(\displaystyle k\)},
xmin=-4.7, xmax=98.7,
xtick style={color=black},
y grid style={darkgray176},
ylabel={Inner Krylov method iterations},
ymin=-6.85, ymax=209.85,
ytick style={color=black}
]
\addplot [semithick, steelblue31119180, mark=*, mark size=2, mark options={solid}]
table {%
0 3
1 5
2 5
3 5
4 6
5 8
6 9
7 12
8 11
9 13
10 14
11 17
12 19
13 24
14 27
15 31
16 34
17 37
18 41
19 44
20 48
21 45
22 40
23 38
24 42
25 45
26 44
27 47
28 45
29 48
30 48
31 50
32 51
33 52
34 54
35 60
36 56
37 51
38 53
39 55
40 59
41 52
42 52
43 49
44 50
45 50
46 51
47 51
48 50
49 50
50 51
51 48
52 46
53 47
54 47
55 46
56 45
57 45
58 44
59 48
60 47
};
\addlegendentry{$N = 40$}
\addplot [semithick, black, mark=x, mark size=2, mark options={solid}, only marks, forget plot]
table {%
20 48
35 60
40 59
};
\addplot [semithick, darkorange25512714, mark=*, mark size=2, mark options={solid}]
table {%
0 3
1 5
2 5
3 5
4 6
5 8
6 9
7 11
8 14
9 14
10 15
11 16
12 19
13 21
14 26
15 31
16 34
17 38
18 41
19 44
20 46
21 50
22 50
23 49
24 48
25 49
26 50
27 52
28 53
29 53
30 57
31 61
32 67
33 62
34 62
35 60
36 57
37 58
38 96
39 60
40 56
41 65
42 66
43 66
44 60
45 61
46 62
47 61
48 61
49 60
50 58
51 57
52 55
53 54
54 54
55 52
56 51
57 51
58 51
59 50
60 50
};
\addlegendentry{$N = 80$}
\addplot [semithick, black, mark=x, mark size=2, mark options={solid}, only marks, forget plot]
table {%
32 67
35 60
38 96
};
\addplot [semithick, forestgreen4416044, mark=*, mark size=2, mark options={solid}]
table {%
0 3
1 5
2 5
3 5
4 6
5 8
6 9
7 11
8 12
9 15
10 17
11 17
12 20
13 22
14 25
15 30
16 33
17 36
18 40
19 43
20 45
21 47
22 50
23 50
24 50
25 50
26 50
27 52
28 53
29 53
30 57
31 60
32 63
33 59
34 62
35 66
36 78
37 68
38 60
39 61
40 61
41 61
42 63
43 64
44 62
45 62
46 60
47 59
48 58
49 58
50 57
51 54
52 54
53 51
54 50
55 49
};
\addlegendentry{$N = 160$}
\addplot [semithick, black, mark=x, mark size=2, mark options={solid}, only marks, forget plot]
table {%
32 63
36 78
};
\addplot [semithick, crimson2143940, mark=*, mark size=2, mark options={solid}]
table {%
0 3
1 5
2 5
3 5
4 7
5 8
6 9
7 11
8 12
9 15
10 17
11 20
12 20
13 23
14 25
15 29
16 32
17 35
18 39
19 42
20 45
21 47
22 49
23 50
24 50
25 50
26 50
27 51
28 53
29 54
30 57
31 59
32 63
33 153
34 151
35 88
36 191
37 156
38 182
39 161
40 141
41 123
42 93
43 79
44 61
45 50
46 36
47 38
48 30
49 20
50 23
51 25
52 23
53 30
54 33
55 38
56 51
57 58
58 65
59 73
60 85
61 83
62 200
63 200
64 200
65 200
66 200
67 200
68 165
69 67
70 52
71 47
72 46
73 50
74 52
75 54
76 56
77 58
78 60
79 63
80 63
81 63
82 65
83 64
84 61
85 61
86 61
87 60
88 59
89 55
90 54
91 51
92 50
93 48
94 49
};
\addlegendentry{$N = 320$}
\addplot [semithick, black, mark=x, mark size=2, mark options={solid}, only marks, forget plot]
table {%
33 153
37 156
38 182
39 161
40 141
41 123
42 93
43 79
44 61
45 50
48 30
62 200
63 200
64 200
65 200
66 200
67 200
};
\addplot [semithick, mediumpurple148103189, mark=*, mark size=2, mark options={solid}]
table {%
0 3
1 5
2 5
3 5
4 7
5 8
6 9
7 11
8 12
9 15
10 17
11 20
12 23
13 24
14 26
15 30
16 32
17 35
18 38
19 42
20 44
21 46
22 48
23 49
24 50
25 49
26 50
27 50
28 53
29 54
30 56
31 59
32 62
33 178
34 165
35 142
36 68
37 200
38 159
39 70
40 61
41 60
42 61
43 61
44 63
45 64
46 64
47 63
48 63
49 60
50 61
51 61
52 60
53 60
54 60
55 57
56 54
57 51
58 49
59 48
};
\addlegendentry{$N = 640$}
\addplot [semithick, black, mark=x, mark size=2, mark options={solid}, only marks, forget plot]
table {%
33 178
37 200
};
\end{axis}

\end{tikzpicture}
  \vspace{-1.5em}
  \caption{Krylov solver iterations: {\scshape gmres} with the block-triangular decomposition-free Schur complement preconditioner.}
  \label{fig:ksp_iters_gmres_block_triangular}
\end{figure}
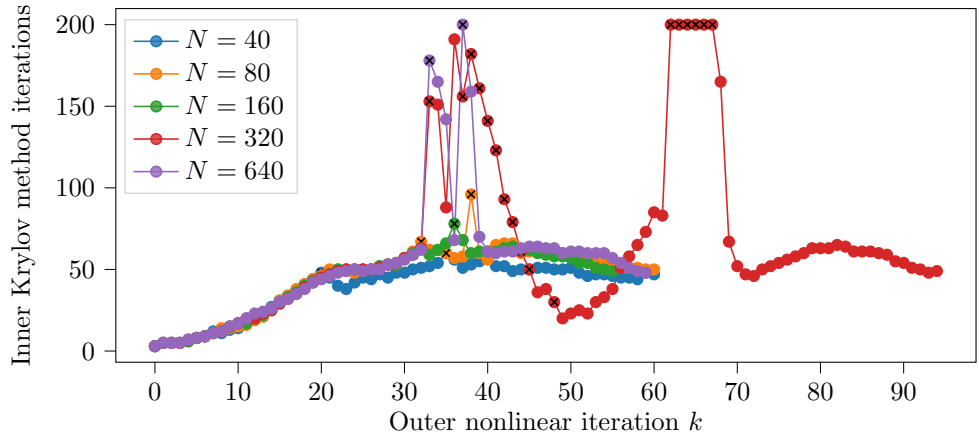

\begin{figure}[tbp]
  \centering
\begin{tikzpicture}

\definecolor{crimson2143940}{RGB}{214,39,40}
\definecolor{darkgray176}{RGB}{176,176,176}
\definecolor{darkorange25512714}{RGB}{255,127,14}
\definecolor{forestgreen4416044}{RGB}{44,160,44}
\definecolor{lightgray204}{RGB}{204,204,204}
\definecolor{mediumpurple148103189}{RGB}{148,103,189}
\definecolor{steelblue31119180}{RGB}{31,119,180}

\begin{axis}[
height=0.3\textheight,
legend cell align={left},
legend style={
  fill opacity=0.8,
  draw opacity=1,
  text opacity=1,
  at={(0.01,0.97)},
  anchor=north west,
  draw=lightgray204
},
tick align=outside,
tick pos=left,
width=\textwidth,
x grid style={darkgray176},
xlabel={Outer nonlinear iteration \(\displaystyle k\)},
xmin=-4.7, xmax=98.7,
xtick style={color=black},
y grid style={darkgray176},
ylabel={Inner Krylov method iterations},
ymin=-8.95, ymax=209.95,
ytick style={color=black}
]
\addplot [semithick, steelblue31119180, mark=*, mark size=2, mark options={solid}]
table {%
0 1
1 1
2 1
3 1
4 1
5 1
6 1
7 1
8 1
9 1
10 1
11 1
12 1
13 1
14 1
15 1
16 1
17 1
18 1
19 1
20 1
21 1
22 1
23 1
24 1
25 1
26 1
27 1
28 1
29 1
30 1
31 1
32 1
33 1
34 1
35 1
36 1
37 1
38 1
39 1
40 1
41 1
42 1
43 1
44 1
45 1
46 1
47 1
48 1
49 1
50 1
51 1
52 1
53 1
54 1
55 1
56 1
57 1
58 1
59 1
60 1
61 1
62 1
63 1
64 1
65 1
66 1
67 1
68 1
69 1
70 1
71 1
72 1
73 1
74 1
75 1
76 1
77 1
78 1
79 1
80 1
81 1
82 1
};
\addlegendentry{Direct}
\addplot [semithick, black, mark=x, mark size=2, mark options={solid}, only marks, forget plot]
table {%
33 1
37 1
38 1
39 1
40 1
41 1
42 1
43 1
44 1
45 1
47 1
48 1
};
\addplot [semithick, darkorange25512714, mark=*, mark size=2, mark options={solid}]
table {%
0 3
1 9
2 8
3 10
4 10
5 12
6 13
7 15
8 16
9 17
10 18
11 22
12 20
13 21
14 19
15 21
16 23
17 24
18 26
19 28
20 29
21 37
22 73
23 75
24 76
25 76
26 78
27 78
28 33
29 31
30 33
31 35
32 37
33 54
34 69
35 51
36 49
37 41
38 40
39 40
40 33
41 35
42 37
43 37
44 37
45 35
46 37
47 36
48 37
49 38
50 40
51 40
52 42
53 41
54 43
55 36
56 46
57 108
58 99
59 97
};
\addlegendentry{Matching}
\addplot [semithick, black, mark=x, mark size=2, mark options={solid}, only marks, forget plot]
table {%
33 54
34 69
};
\addplot [semithick, forestgreen4416044, mark=*, mark size=2, mark options={solid}]
table {%
0 3
1 9
2 9
3 10
4 10
5 12
6 13
7 15
8 16
9 18
10 18
11 22
12 21
13 21
14 19
15 21
16 22
17 24
18 26
19 28
20 29
21 37
22 73
23 74
24 75
25 76
26 78
27 78
28 33
29 31
30 33
31 35
32 37
33 54
34 53
35 50
36 46
37 41
38 40
39 40
40 43
41 54
42 49
43 43
44 40
45 38
46 37
47 37
48 35
49 38
50 34
51 38
52 40
53 40
54 38
55 39
56 41
57 42
58 46
59 43
60 109
61 102
62 100
63 99
};
\addlegendentry{AMG[$S_1$] matching}
\addplot [semithick, black, mark=x, mark size=2, mark options={solid}, only marks, forget plot]
table {%
33 54
34 53
40 43
};
\addplot [semithick, crimson2143940, mark=*, mark size=2, mark options={solid}]
table {%
0 3
1 9
2 9
3 10
4 10
5 12
6 13
7 15
8 16
9 18
10 18
11 22
12 21
13 21
14 19
15 21
16 22
17 24
18 26
19 28
20 29
21 37
22 73
23 74
24 76
25 77
26 78
27 80
28 33
29 31
30 33
31 36
32 38
33 67
34 77
35 119
36 78
37 57
38 49
39 43
40 48
41 92
42 90
43 129
44 108
45 101
46 93
47 88
48 71
49 60
50 47
51 39
52 41
53 39
54 39
55 38
56 40
57 34
58 85
59 116
60 88
61 96
62 91
63 81
64 93
65 78
66 56
67 47
68 38
69 37
70 38
71 38
72 36
73 40
74 37
75 37
76 36
77 36
78 36
79 38
80 38
81 40
82 40
83 42
84 42
85 41
86 43
87 111
88 102
89 100
90 98
};
\addlegendentry{Decomposition-free}
\addplot [semithick, black, mark=x, mark size=2, mark options={solid}, only marks, forget plot]
table {%
33 67
35 119
36 78
41 92
42 90
43 129
44 108
46 93
47 88
58 85
59 116
60 88
61 96
62 91
64 93
};
\addplot [semithick, mediumpurple148103189, mark=*, mark size=2, mark options={solid}]
table {%
0 3
1 5
2 5
3 5
4 7
5 8
6 9
7 11
8 12
9 15
10 17
11 20
12 20
13 23
14 25
15 29
16 32
17 35
18 39
19 42
20 45
21 47
22 49
23 50
24 50
25 50
26 50
27 51
28 53
29 54
30 57
31 59
32 63
33 153
34 151
35 88
36 191
37 156
38 182
39 161
40 141
41 123
42 93
43 79
44 61
45 50
46 36
47 38
48 30
49 20
50 23
51 25
52 23
53 30
54 33
55 38
56 51
57 58
58 65
59 73
60 85
61 83
62 200
63 200
64 200
65 200
66 200
67 200
68 165
69 67
70 52
71 47
72 46
73 50
74 52
75 54
76 56
77 58
78 60
79 63
80 63
81 63
82 65
83 64
84 61
85 61
86 61
87 60
88 59
89 55
90 54
91 51
92 50
93 48
94 49
};
\addlegendentry{Block-triangular}
\addplot [semithick, black, mark=x, mark size=2, mark options={solid}, only marks, forget plot]
table {%
33 153
37 156
38 182
39 161
40 141
41 123
42 93
43 79
44 61
45 50
48 30
62 200
63 200
64 200
65 200
66 200
67 200
};
\end{axis}

\end{tikzpicture}
  \vspace{-1.5em}
  \caption{Krylov solver iterations: Comparison of different variants for $N=320$.}
  \label{fig:ksp_iters_variants_320}
\end{figure}

\begin{figure}[tbp]
  \centering
\begin{tikzpicture}

\definecolor{crimson2143940}{RGB}{214,39,40}
\definecolor{darkgray176}{RGB}{176,176,176}
\definecolor{darkorange25512714}{RGB}{255,127,14}
\definecolor{forestgreen4416044}{RGB}{44,160,44}
\definecolor{lightgray204}{RGB}{204,204,204}
\definecolor{mediumpurple148103189}{RGB}{148,103,189}
\definecolor{steelblue31119180}{RGB}{31,119,180}

\begin{axis}[
height=0.3\textheight,
legend cell align={left},
legend style={fill opacity=0.8, draw opacity=1, text opacity=1, at={(0.01, 0.97)}, anchor=north west, draw=lightgray204},
tick align=outside,
tick pos=left,
width=\textwidth,
x grid style={darkgray176},
xlabel={Outer nonlinear iteration \(\displaystyle k\)},
xmin=-3.45, xmax=72.45,
xtick style={color=black},
y grid style={darkgray176},
ylabel={Inner Krylov method iterations},
ymin=-8.95, ymax=209.95,
ytick style={color=black}
]
\addplot [semithick, steelblue31119180, mark=*, mark size=2, mark options={solid}]
table {%
0 1
1 1
2 1
3 1
4 1
5 1
6 1
7 1
8 1
9 1
10 1
11 1
12 1
13 1
14 1
15 1
16 1
17 1
18 1
19 1
20 1
21 1
22 1
23 1
24 1
25 1
26 1
27 1
28 1
29 1
30 1
31 1
32 1
33 1
34 1
35 1
36 1
37 1
38 1
39 1
40 1
41 1
42 1
43 1
44 1
45 1
46 1
47 1
48 1
49 1
50 1
51 1
52 1
53 1
54 1
55 1
56 1
57 1
};
\addlegendentry{Direct}
\addplot [semithick, black, mark=x, mark size=2, mark options={solid}, only marks, forget plot]
table {%
33 1
};
\addplot [semithick, darkorange25512714, mark=*, mark size=2, mark options={solid}]
table {%
0 3
1 9
2 8
3 8
4 8
5 10
6 12
7 13
8 15
9 17
10 17
11 18
12 19
13 18
14 18
15 18
16 19
17 21
18 21
19 23
20 26
21 27
22 73
23 75
24 76
25 76
26 78
27 78
28 28
29 27
30 28
31 30
32 32
33 46
34 63
35 48
36 43
37 37
38 35
39 33
40 32
41 28
42 32
43 32
44 32
45 33
46 33
47 33
48 29
49 40
50 38
51 40
52 42
53 42
54 42
55 41
56 38
57 41
58 43
59 53
};
\addlegendentry{Matching}
\addplot [semithick, black, mark=x, mark size=2, mark options={solid}, only marks, forget plot]
table {%
33 46
34 63
};
\addplot [semithick, forestgreen4416044, mark=*, mark size=2, mark options={solid}]
table {%
0 3
1 9
2 8
3 8
4 8
5 10
6 12
7 13
8 15
9 17
10 18
11 18
12 18
13 18
14 18
15 18
16 19
17 21
18 21
19 23
20 26
21 27
22 73
23 75
24 76
25 74
26 76
27 78
28 28
29 27
30 28
31 30
32 32
33 48
34 64
35 51
36 43
37 38
38 35
39 33
40 32
41 32
42 32
43 30
44 30
45 33
46 34
47 29
48 35
49 35
50 35
51 37
52 44
53 32
54 39
55 41
56 36
57 35
58 43
59 43
60 50
61 33
62 45
};
\addlegendentry{AMG[$S_1$] matching}
\addplot [semithick, black, mark=x, mark size=2, mark options={solid}, only marks, forget plot]
table {%
33 48
34 64
57 35
};
\addplot [semithick, crimson2143940, mark=*, mark size=2, mark options={solid}]
table {%
0 3
1 9
2 8
3 8
4 8
5 10
6 12
7 13
8 15
9 17
10 18
11 18
12 19
13 18
14 18
15 18
16 19
17 21
18 21
19 23
20 26
21 27
22 74
23 76
24 76
25 76
26 78
27 78
28 28
29 27
30 28
31 30
32 33
33 56
34 77
35 64
36 43
37 37
38 37
39 35
40 32
41 32
42 32
43 32
44 27
45 34
46 33
47 35
48 35
49 37
50 32
51 39
52 42
53 34
54 28
55 21
56 32
57 39
58 39
59 41
60 75
61 31
62 31
63 37
64 41
65 56
66 103
67 97
68 95
69 94
};
\addlegendentry{Decomposition-free}
\addplot [semithick, black, mark=x, mark size=2, mark options={solid}, only marks, forget plot]
table {%
33 56
34 77
55 21
56 32
57 39
62 31
63 37
};
\addplot [semithick, mediumpurple148103189, mark=*, mark size=2, mark options={solid}]
table {%
0 3
1 5
2 5
3 5
4 7
5 8
6 9
7 11
8 12
9 15
10 17
11 20
12 23
13 24
14 26
15 30
16 32
17 35
18 38
19 42
20 44
21 46
22 48
23 49
24 50
25 49
26 50
27 50
28 53
29 54
30 56
31 59
32 62
33 178
34 165
35 142
36 68
37 200
38 159
39 70
40 61
41 60
42 61
43 61
44 63
45 64
46 64
47 63
48 63
49 60
50 61
51 61
52 60
53 60
54 60
55 57
56 54
57 51
58 49
59 48
};
\addlegendentry{Block-triangular}
\addplot [semithick, black, mark=x, mark size=2, mark options={solid}, only marks, forget plot]
table {%
33 178
37 200
};
\end{axis}

\end{tikzpicture}
  \vspace{-1.5em}
  \caption{Krylov solver iterations: Comparison of different variants for $N=640$.}
  \label{fig:ksp_iters_variants_640}
\end{figure}

\begin{figure}[tbp]
  \centering
\begin{tikzpicture}

\definecolor{crimson2143940}{RGB}{214,39,40}
\definecolor{darkgray176}{RGB}{176,176,176}
\definecolor{darkorange25512714}{RGB}{255,127,14}
\definecolor{forestgreen4416044}{RGB}{44,160,44}
\definecolor{lightgray204}{RGB}{204,204,204}
\definecolor{mediumpurple148103189}{RGB}{148,103,189}
\definecolor{steelblue31119180}{RGB}{31,119,180}

\begin{axis}[
height=0.25\textheight,
legend cell align={left},
legend style={
  fill opacity=0.8,
  draw opacity=1,
  text opacity=1,
  at={(0.01,0.97)},
  anchor=north west,
  draw=lightgray204
},
tick align=outside,
tick pos=left,
width=\textwidth,
x grid style={darkgray176},
xlabel={Iteration \(\displaystyle k\)},
xmin=-3.45, xmax=72.45,
xtick style={color=black},
y grid style={darkgray176},
ylabel={Wall clock time [s]},
ymin=-641.6569734, ymax=13786.1275274,
ytick style={color=black}
]
\addplot [semithick, steelblue31119180, mark=*, mark size=2, mark options={solid}]
table {%
0 40.270953
1 81.51909
2 122.859301
3 164.090464
4 205.350042
5 246.634044
6 287.979921
7 329.304329
8 370.581275
9 411.924065
10 453.340611
11 494.837003
12 536.155657
13 577.838082
14 619.358307
15 661.129187
16 703.01088
17 744.963641
18 786.041323
19 827.072924
20 867.949678
21 908.760644
22 949.699285
23 990.47041
24 1031.378512
25 1072.133915
26 1113.089073
27 1153.956349
28 1194.803994
29 1235.542839
30 1276.546741
31 1317.526695
32 1358.379115
33 1399.428663
34 1436.674711
35 1477.588457
36 1518.478309
37 1559.536979
38 1600.420476
39 1641.248804
40 1682.112317
41 1722.878676
42 1763.93574
43 1804.701986
44 1845.531831
45 1886.411862
46 1927.226215
47 1967.911277
48 2008.820537
49 2049.593864
50 2090.187018
51 2130.644349
52 2171.321464
53 2211.602355
54 2251.826247
55 2291.891453
56 2332.039796
57 2354.041123
};
\addlegendentry{Direct}
\addplot [semithick, darkorange25512714, mark=*, mark size=2, mark options={solid}]
table {%
0 48.949913
1 130.317042
2 205.468443
3 281.036921
4 356.78999
5 441.753753
6 538.610402
7 640.887174
8 753.259141
9 876.197614
10 999.3724
11 1127.428749
12 1261.123136
13 1389.404485
14 1517.896738
15 1646.604029
16 1778.404548
17 1923.027574
18 2065.144277
19 2220.87275
20 2391.471107
21 2563.770561
22 2972.883245
23 3402.424912
24 3844.971582
25 4304.827626
26 4776.821141
27 5249.526161
28 5446.017729
29 5634.059799
30 5832.813151
31 6041.837934
32 6259.914121
33 6561.446648
34 6953.686778
35 7262.370735
36 7548.709012
37 7797.60877
38 8035.84966
39 8265.089967
40 8482.5452
41 8675.665671
42 8899.607262
43 9122.008312
44 9344.890689
45 9572.627319
46 9803.091679
47 10032.409227
48 10237.410852
49 10501.233358
50 10757.541901
51 11027.470129
52 11308.827012
53 11587.331914
54 11872.008747
55 12148.237531
56 12399.344579
57 12673.43885
58 12955.22955
59 13130.319141
};
\addlegendentry{Matching}
\addplot [semithick, forestgreen4416044, mark=*, mark size=2, mark options={solid}]
table {%
0 39.34689
1 104.734311
2 165.522392
3 225.984502
4 286.471839
5 355.008236
6 432.541293
7 513.725535
8 603.604478
9 701.771162
10 801.312513
11 901.557653
12 1001.585171
13 1101.966301
14 1201.918962
15 1301.508731
16 1404.882374
17 1516.532199
18 1628.078637
19 1747.589954
20 1878.598251
21 2014.04801
22 2332.434387
23 2653.40213
24 2979.229649
25 3291.261486
26 3617.795518
27 3966.153036
28 4115.968894
29 4261.747357
30 4411.641397
31 4568.063482
32 4735.142028
33 4973.51947
34 5276.773376
35 5526.662158
36 5745.161258
37 5941.627906
38 6122.098469
39 6290.882546
40 6457.051783
41 6625.430254
42 6794.081785
43 6957.82603
44 7116.287919
45 7292.700358
46 7472.567809
47 7632.004651
48 7811.559067
49 7997.925264
50 8181.682833
51 8375.799406
52 8596.921369
53 8769.048978
54 8972.182434
55 9185.146171
56 9373.314488
57 9558.683108
58 9775.676287
59 9996.463999
60 10247.930515
61 10419.862211
62 10538.402684
};
\addlegendentry{AMG[$S_1$] matching}
\addplot [semithick, crimson2143940, mark=*, mark size=2, mark options={solid}]
table {%
0 14.151413
1 33.282963
2 51.727509
3 70.146672
4 88.507484
5 108.347216
6 129.774191
7 151.706542
8 175.172323
9 200.066873
10 225.905072
11 251.591341
12 278.094544
13 303.872636
14 329.564728
15 355.546086
16 382.08377
17 410.122674
18 438.052865
19 467.785563
20 499.578602
21 532.095957
22 599.378659
23 666.786686
24 734.35761
25 798.271747
26 863.491807
27 928.430476
28 959.574259
29 989.764149
30 1020.988448
31 1053.609946
32 1088.153807
33 1138.248691
34 1198.957808
35 1250.900444
36 1292.089462
37 1329.099673
38 1366.410183
39 1402.221858
40 1436.005284
41 1469.995536
42 1504.016559
43 1537.472097
44 1567.9743
45 1603.021959
46 1637.582694
47 1673.574956
48 1709.143417
49 1746.190352
50 1780.230269
51 1818.65415
52 1859.252445
53 1894.610511
54 1925.364767
55 1951.618552
56 1981.698354
57 2016.669349
58 2051.896022
59 2091.701895
60 2154.493378
61 2187.444474
62 2220.58027
63 2254.05791
64 2290.192223
65 2339.959401
66 2420.633593
67 2498.116342
68 2573.61077
69 2613.347499
};
\addlegendentry{Decomposition-free}
\addplot [semithick, mediumpurple148103189, mark=*, mark size=2, mark options={solid}]
table {%
0 15.321007
1 32.364273
2 49.435657
3 66.530275
4 85.08605
5 104.348612
6 124.55976
7 146.233967
8 168.961536
9 194.046461
10 220.626876
11 249.72776
12 281.480861
13 314.058193
14 348.205633
15 385.788545
16 425.233509
17 467.422952
18 512.04271
19 560.302124
20 610.090744
21 662.566946
22 716.168754
23 769.673659
24 823.576061
25 874.65189
26 927.300639
27 979.823683
28 1034.692337
29 1090.510608
30 1148.345797
31 1208.513277
32 1271.842285
33 1436.730691
34 1585.625585
35 1718.798887
36 1787.140029
37 1882.728469
38 2027.168992
39 2100.942368
40 2164.007995
41 2225.970803
42 2289.325624
43 2351.448791
44 2414.481375
45 2478.860265
46 2542.76911
47 2606.449272
48 2669.932152
49 2730.894385
50 2792.431688
51 2854.103428
52 2914.571865
53 2974.939423
54 3034.951706
55 3092.683925
56 3147.635364
57 3199.779401
58 3250.311589
59 3277.286796
};
\addlegendentry{Block-triangular}
\end{axis}

\end{tikzpicture}
  \vspace{-1.5em}
  \caption{Comparison of wall clock time vs.~iterations of different variants for $N=640$.}
  \label{fig:wall_time_comparison_640}
\end{figure}
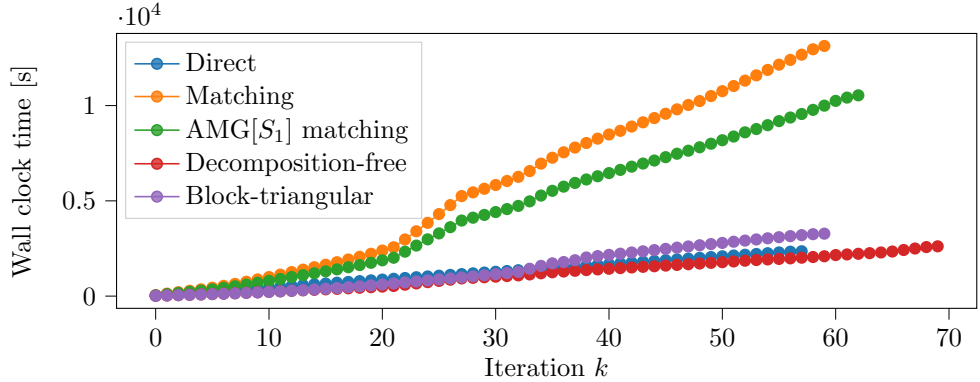

\subsection{Preconditioner comparison for the 2D case}

We start with the 2D instances. In Fig.~\ref{fig:ksp_iters_minres_exact_S1_S2} we see that for the basic Schur complement preconditioner the number of {\scshape minres} iterations per iteration of the sequential homotopy method stays moderately low (10--40 iterations) most of the time. Between iterations 20 and 30 as well as close to the solution, the number of iterations rises to about 60--80. This shows that the method saves numerical effort in the linear algebra part when being far away from the solution. The number of nonlinear iterations increases considerably when going from $N=160$ to $N=320$, which we attribute to the steep boundary layer of the optimal state, which is not faithfully resolved on meshes smaller than $N \le 160$.

In Fig.~\ref{fig:ksp_iters_minres_exact_S1_S2}--\ref{fig:ksp_iters_variants_640}, discarded iterations (due to $\theta > \theta_\mathrm{max}$ or Krylov-subspace methods failing to converge for a given linear system) are marked with an additional black $\times$.

Compared to Fig.~\ref{fig:ksp_iters_minres_exact_S1_S2}, we see in Fig.~\ref{fig:ksp_iters_minres_lu_S1_match_S2} that the matching approach slightly increases the required number of {\scshape minres} iterations, while the number of nonlinear iterations stays roughly the same. In contrast to the basic Schur complement preconditioner (which takes more than 41 hours for $N=320$), also the case $N=640$ can be solved in under four hours now (cf. Fig.~\ref{fig:wall_time_comparison_640}).

Using the AMG[$\hat{S}_1$] and decomposition-free Schur complement preconditioners result in qualitatively very similar behavior, so we just display the latter in Fig.~\ref{fig:ksp_iters_minres_lu_free}. This exhibits the efficiency and robustness of the solver for a range of problem sizes.

Using the block-triangular version of the decomposition-free Schur complement preconditioner seems to deliver a more robust preconditioner with fewer fluctuations in the number of iterations
(see Fig.~\ref{fig:ksp_iters_gmres_block_triangular}).

We compare the number of required inner and outer iterations resulting from the different preconditioner variants on grids with fixed size in Fig.~\ref{fig:ksp_iters_variants_320} ($N=320$) and Fig.~\ref{fig:ksp_iters_variants_640} ($N=640$). We observe that the expected increase in the number of outer iterations for solving the linear subproblems inexactly with Krylov-subspace methods compared to direct decomposition approaches is moderate for $N=640$ (70 or better vs.~57). For $N=320$, some preconditioner choices perform better for the numerical challenges that start in outer iteration $k=33$, leading to fewer overall outer iterations. The differences are mostly due to the number of failed iterations (marked with $\times$), whenever the sequential homotopy method is forced to increase $\lambda$ to keep inside the region of local convergence for the nonlinear subproblems.

When it comes to runtime between all preconditioners, there is a clear benefit of using the decomposition-free Schur complement preconditioners,
which are competitive with
the use of direct linear algebra even for the 2D problem on a reasonably fine grid 
(see Fig.~\ref{fig:wall_time_comparison_640}). 

\subsection{Parallelization for the 2D case}

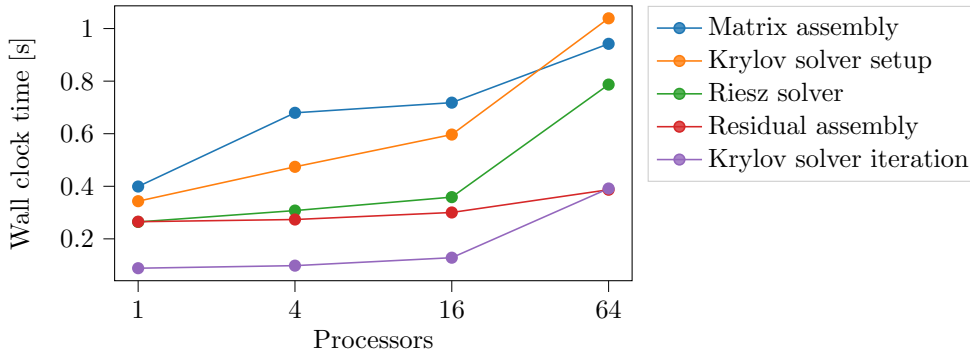
\begin{figure}[tbp]
  \centering
\begin{tikzpicture}

\definecolor{crimson2143940}{RGB}{214,39,40}
\definecolor{darkgray176}{RGB}{176,176,176}
\definecolor{darkorange25512714}{RGB}{255,127,14}
\definecolor{forestgreen4416044}{RGB}{44,160,44}
\definecolor{lightgray204}{RGB}{204,204,204}
\definecolor{mediumpurple148103189}{RGB}{148,103,189}
\definecolor{steelblue31119180}{RGB}{31,119,180}

\begin{axis}[
height=0.25\textheight,
legend cell align={left},
legend style={fill opacity=0.8, draw opacity=1, text opacity=1, draw=white!80!black},
legend pos=outer north east,
log basis x={10},
tick align=outside,
tick pos=left,
width=0.65\textwidth,
x grid style={darkgray176},
xlabel={Processors},
xmin=0.812252396356235, xmax=78.7932424540746,
xmode=log,
xtick=data,
xticklabels={$1$, $4$, $16$, $64$},
xtick style={color=black},
y grid style={darkgray176},
ylabel={Wall clock time [s]},
ymin=0.0408156, ymax=1.0864844,
ytick style={color=black}
]
\addplot [semithick, steelblue31119180, mark=*, mark size=2, mark options={solid}]
table {%
1 0.399111
4 0.679511
16 0.718144
64 0.941928
};
\addlegendentry{Matrix assembly}
\addplot [semithick, darkorange25512714, mark=*, mark size=2, mark options={solid}]
table {%
1 0.343295
4 0.474039
16 0.596765
64 1.038954
};
\addlegendentry{Krylov solver setup}
\addplot [semithick, forestgreen4416044, mark=*, mark size=2, mark options={solid}]
table {%
1 0.264375
4 0.307594
16 0.358727
64 0.786896
};
\addlegendentry{Riesz solver}
\addplot [semithick, crimson2143940, mark=*, mark size=2, mark options={solid}]
table {%
1 0.265266
4 0.273581
16 0.300339
64 0.386916
};
\addlegendentry{Residual assembly}
\addplot [semithick, mediumpurple148103189, mark=*, mark size=2, mark options={solid}]
table {%
1 0.088346
4 0.098106
16 0.128332
64 0.391471
};
\addlegendentry{Krylov solver iteration}
\end{axis}

\end{tikzpicture}
  \vspace{-1.5em}
  \caption{Weak scaling results for 2D case using the block-diagonal decomposition free preconditioner and {\scshape minres}. We show the average wall clock times over the last 10 iterations on the finest grid level.}
  \label{fig:weak_scaling}
\end{figure}

\begin{table}[tbp]
  \centering
  \begin{tabular}{ccccc}
    \hline
    Number of processors & 1 & 4 & 16 & 64 \\
    $N$ & 320 & 640 & 1,280 & 2,560 \\
    Degrees of freedom & 309,123 & 1,232,643 & 4,922,883 & 19,676,163 \\
    \hline
  \end{tabular}
  \caption{Correspondence of number of processors and discretization for the weak scaling experiment.}
  \label{tab:weak_scaling}
\end{table}

We report weak scaling results of the decomposition-free Schur complement preconditioner in Fig.~\ref{fig:weak_scaling} for up to 64 processors. The number of degrees of freedom for each run is provided in Tab.~\ref{tab:weak_scaling}.
In a case of perfect parallel scaling, all lines in Fig.~\ref{fig:weak_scaling} would be horizontal.
There is a noticeable drop in parallel efficiency when the number of processors and degrees of freedom increase. However, this drop goes in conjunction with a drop in parallel efficiency also for the Riesz solver, which is simply AMG-preconditioned {\scshape cg} for the Poisson equation, whose parallel scaling can be considered state-of-the-art.

\subsection{3D case}

\begin{figure}[tbp]
  \centering
  \includegraphics[width=0.8\textwidth]{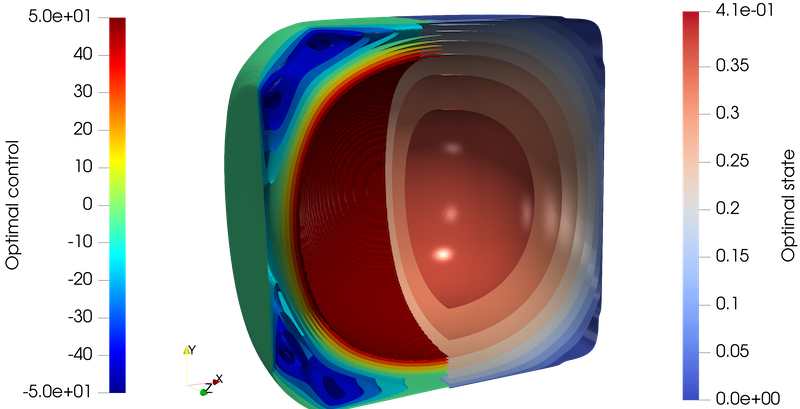}
  \caption{Result for the 3D test case for $N=160$. The unit cube is cut in half in the $X/Y$ plane. On the left we see isosurfaces of the optimal control and on the right isosurfaces of the optimal state.}
  \label{fig:result_3d}
\end{figure}

\begin{table}[tbp]
  \centering
  \begin{tabular}{ccccc}
    $N$ & Degrees of freedom & Outer iter. & {\scshape minres} iter. & Wall clock time [min] \\ \hline
    20  & 27,783     & 52  & 2693 & 1.67   \\
    40  & 206,763    & 39  & 1730 & 2.74   \\
    80  & 1,594,323  & 34  & 1186 & 12.06  \\
    160 & 12,519,843 & 35  & 760  & 63.20  \\
    \hline
  \end{tabular}
  \caption{Statistics for the 3D test case on each discretization level. The computations were performed on 32 processors in parallel.}
  \label{tab:3d}
\end{table}

\begin{figure}[tbp]
  \centering
  \input{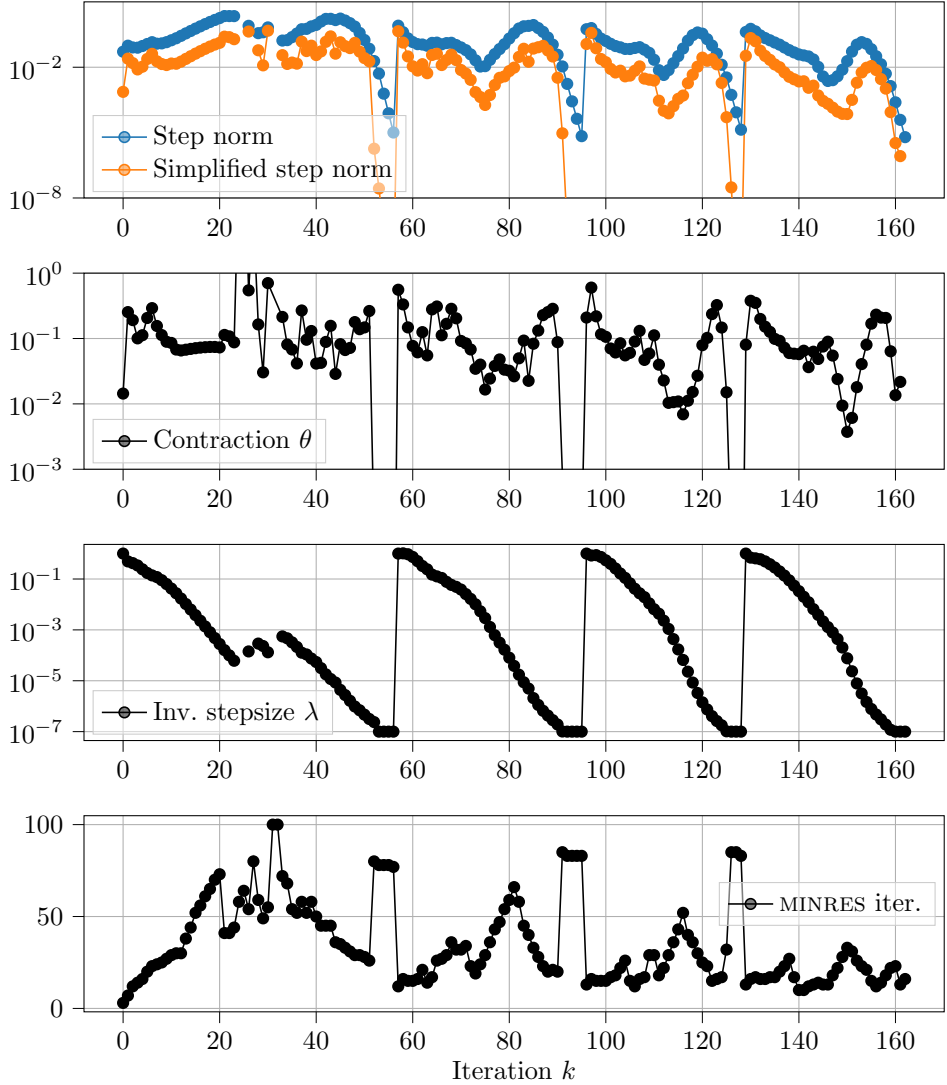}
  \caption{Statistics for the iterations of the 3D test case on refined meshes. Failed iterations are indicated as missing dots in the first and third plot. The contraction $\theta$ is the quotient of the simplified step norm over the step norm.}
  \label{fig:iters_3d}
\end{figure}

\begin{figure}[tbp]
  \centering
  \includegraphics[width=\textwidth]{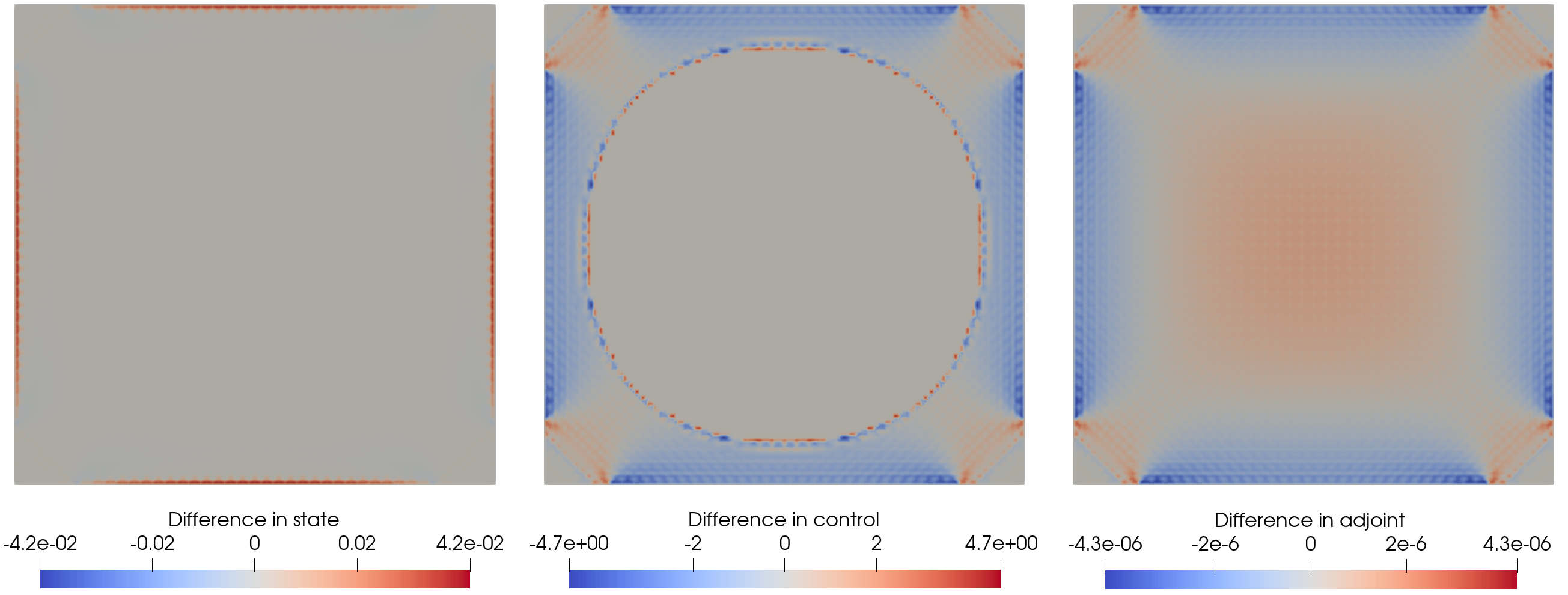}
  \caption{Slice through the X/Y plane of the absolute difference of the optimal state, control, and adjoint approximations on the two finest 3D meshes ($N=160$ minus $N=80$).
  The difference for the control is rather large close to the boundary of the active set (central gray region). The different magnitudes of the differences are partly due to the different numerical ranges (not shown), which are approximately $[0, 0.41]$ (state), $[-50, 50]$ (control), and $[-5 \cdot 10^{-5}, 5 \cdot 10^{-4}]$ (adjoint).}
  \label{fig:diff_refine}
\end{figure}

In Fig.~\ref{fig:result_3d} we show the optimal solution of the model problem~\eqref{eqn:model_problem} on the unit cube on a regular tetrahedral grid with $N = 160$ element edges per edge of the unit cube and control bounds of $\pm 50$. 
Similar to the optimal solutions in 2D, the optimal controls and optimal states exhibit very steep slopes.
The computations are carried out on a sequence of uniformly refined equidistant meshes for $N = 20, 40, 80, 160$. Starting with an initial guess of zero, we employ Alg.~\ref{alg:iseqhom} with the block-diagonal decomposition-free preconditioner and the same tolerance of $\mathrm{tol} = 10^{-5}$ as in the 2D case on each level, moving from the coarsest level to the finest while using the solution of the previous level as the initial guess on the next (finer) mesh. We have slightly reduced the augmentation parameter for the 3D test case from $\rho=10^{-1}$ to $\rho = 10^{-3}$ to avoid occasionally failing iterations on the finer levels. In Tab.~\ref{tab:3d}, we see that most CPU time is spent on the finest level and that comparably few {\scshape minres} iterations are required there. We depict further information about the iterations on the different levels in Fig.~\ref{fig:iters_3d}. From the contraction rates $\theta$ of the monotonicity test, we see that for the initial guesses on each level a rather large $\lambda$ is required to stay in the region of local convergence for the simplified semismooth Newton method, while later, fast local convergence with $\lambda=\lambda_\mathrm{min}$ can be achieved. Rather large increases in the step norms after refinement of the grid can be observed. We attribute these to the spatial resolution of the solution's local features, such as steep slopes at the domain boundary and the boundary of the active set of the solution (cf. Fig.~\ref{fig:result_3d} and Fig.~\ref{fig:diff_refine}), which are very sensitive to the chosen finite element mesh.
Altogether the proposed inexact sequential homotopy method with the suggested preconditioners can reliably and efficiently solve large-scale, highly nonlinear, badly conditioned problems.

\section{Conclusion}

We have extended a sequential homotopy method to allow the use of inexact solvers for the linearized subsystems. We provided analysis for symmetric positive definite, block-diagonal Schur complement preconditioners for double saddle-point systems, with a view to applying these to a prominent class of nonlinear PDE-constrained optimization problems. For a challenging model problem, we provided approximations of the Schur complements for the efficient and parallel application of approximated double saddle-point preconditioners. The implicit regularization feature of the sequential homotopy method was beneficial for the solution of the linearized subproblems with preconditioned {\scshape minres} and {\scshape gmres}. We provided numerical results for a hierarchy of Schur complement approximations, which shed light on the consequences of each approximation step on the way to the eventual fast, effective, and parallelizable, decomposition-free preconditioner.
We provided a weak scaling analysis for the 2D case and efficiently solved a large 3D problem instance with 12 million degrees of freedom.

\section*{Acknowledgements} JWP gratefully acknowledges support from the Engineering and Physical Sciences Research Council (UK) grant EP/S027785/1, and a Fellowship of The Alan Turing Institute.
We thank the unknown referees for their constructive suggestions, which helped to improve an earlier version of this article.

\bibliographystyle{plain}
\bibliography{refs_R1}

\end{document}